\documentclass{article}
\usepackage[utf8]{inputenc}

\usepackage{amsmath,amsthm,amssymb}
\usepackage{times}
\usepackage{enumerate}
\usepackage{amssymb,enumitem,mathtools}
\usepackage{url}
\usepackage{fullpage}
\usepackage{hyperref}
\usepackage{graphicx}
\usepackage{array}

\newtheorem{theorem}{Theorem}[section]
\newtheorem{lemma}[theorem]{Lemma}
\newtheorem{proposition}[theorem]{Proposition}
\newtheorem{corollary}[theorem]{Corollary}
\theoremstyle{definition}

\newtheorem{question}[theorem]{Question}

\theoremstyle{remark}
\newtheorem{remark}[theorem]{Remark}
\newtheorem{conjecture}[theorem]{Conjecture}

\title{Hermitian matrices of roots of unity and their  \\ characteristic polynomials}
\author{Gary R.W. Greaves\thanks{The first author was supported in part by the Singapore Ministry of Education Academic Research Fund (Tier 1);  grant numbers: RG21/20 and RG23/20.} \ \ and Chin Jian Woo  \\ \\
School of Physical and Mathematical Sciences, \\
  Nanyang Technological University, \\
   21 Nanyang Link, Singapore 637371\\
  {\tt gary@ntu.edu.sg} \\ {\tt haysonwoo@gmail.com} }
\date{}

\begin{document}

\maketitle

\begin{abstract}
    We investigate spectral conditions on Hermitian matrices of roots of unity.
    Our main results are conjecturally sharp upper bounds on the number of residue classes of the characteristic polynomial of such matrices modulo ideals generated by powers of $(1-\zeta)$, where $\zeta$ is a root of unity.
    We also prove a generalisation of a classical result of Harary and Schwenk about a relation for traces of powers of a graph-adjacency matrix, which is a crucial ingredient for the proofs of our main results.
\end{abstract}

\section{Introduction}

Throughout, we denote the identity matrix by $I$ and the all-ones matrix by $J$, the order of $I$ and $J$ will be understood from context.
Fix an integer $q \geqslant 2$.
Let $\zeta$ be a primitive $q$th root of unity and let $\mathcal C_q = \langle \zeta \rangle$ be the set of powers of $\zeta$.
Denote by $\mathcal H_n(q)$ the set of all Hermitian $\mathcal C_q$-matrices of order $n$, i.e., the set of Hermitian matrices whose entries are powers of $\zeta$.
Matrices from the set $\mathcal H_n(q)$ play a central role in various areas of mathematics, including discrete geometry~\cite{lemmens73}, quantum information theory~\cite{stacey}, group theory~\cite{taylor}, and wireless communication~\cite{heath}.
One pertinent example where such matrices make an appearance is as signature matrices of equiangular tight frames~\cite{heath}.

Let $B$ denote a complex $d \times n$ matrix having the vectors $\mathbf x_1, \dots, \mathbf x_n$ as its columns and let $B^*$ denote the complex transpose of $B$.
The set $F = \{ \mathbf x_1, \dots, \mathbf x_n \}$ is called an $(n,d)$-\textbf{equiangular tight frame} (ETF) if
 $B^*B$ has $1$s on its diagonal, all entries on the off-diagonal have the same absolute value, and $BB^* = nI/d$.
ETFs are extremal objects with respect to the Welch bound, which says that, for any set $F = \{ \mathbf x_1, \dots, \mathbf x_n \}$ of unit vectors in $\mathbb C^d$, we have
\[
	\max_{i \ne j} | \mathbf x_i^* \mathbf x_j  | \geqslant \sqrt{ \frac{n-d}{ d(n-1)}}.
\]
Equality holds in the Welch bound if and only if $F$ is an ETF.
We can obtain the \textbf{signature matrix} $S(F)$ of an ETF from the matrix $B$ as
\[
S(F) := \sqrt{ \frac{d(n-1)}{n-d}}(B^*B -I).
\]
Note that the off-diagonal entries of a signature matrix of an ETF all lie on the unit circle.
We call an Hermitian matrix $S$ a $q$\textbf{-Seidel matrix} if all of its off-diagonal entries are in $\mathcal C_q$ and its diagonal entries are all $0$. 
Our $2$-Seidel matrices are known simply as \textit{Seidel matrices} in the literature, see for example~\cite{ghorbani, GG18, GKMS16,roux,SzOs18}.
Interestingly, for most known constructions of ETFs~\cite{Fi0}, the signature matrix is a $q$-Seidel matrix for some $q \in \mathbb N$.
For this reason, ETFs having signature matrices whose (off-diagonal) entries are $q$th roots of unity have been investigated~\cite{pthRoot, thirdRoot,fourthRoot} when $q$ is a prime power.
We extend the above investigations by studying the spectral properties of $q$-Seidel matrices that do not necessarily correspond to ETFs.

An ETF is called \textbf{real} if the vectors $\mathbf x_i \in \mathbb R^d$ for all $i \in \{1,\dots,n\}$.
Azarija and Marc~\cite{Azarija75,Azarija95} recently showed that there exists neither a real $(76,19)$-ETF nor a real $(96,20)$-ETF.
That is, there does not exist a $(76,19)$-ETF nor a $(96,20)$-ETF whose signature matrix is a $2$-Seidel matrix.
Fickus et al.~\cite{Steiner12} showed that there does exist a $(96,20)$-ETF whose signature matrix is a $4$-Seidel matrix, but it is currently not known if there exists a $(76,19)$-ETF whose signature matrix is a $4$-Seidel matrix.
However, it is known that there exists a $(76,19)$-ETF whose signature matrix is a $6$-Seidel matrix~\cite{hyperOvals16}.
Thus, our main motivation stems from the following question. 
\begin{question}
\label{que:main}
For a given pair $(n,d)$, for which $q\in \mathbb N$ can there exist a $q$-Seidel matrix that is a signature matrix of an $(n,d)$-ETF?
\end{question}

Instead of restricting ourselves to $q$-Seidel matrices, we work more generally with the sets $\mathcal H_n(q)$ of Hermitian $\mathcal C_q$-matrices of order $n$ and investigate their spectral properties.
Note that we obtain an element of $\mathcal H_n(q)$ by adding the identity matrix to any $q$-Seidel matrix of order $n$.
Our main results are necessary conditions on the coefficients of the characteristic polynomial $\operatorname{Char}_H(x) := \det(xI-H)$ of a matrix $H \in \mathcal H_n(q)$.
Our investigation is a natural follow-up to \cite{GreavesYatsyna19}, where certain necessary spectral conditions for Seidel matrices were introduced.
In particular, Greaves and Yatsyna~\cite{GreavesYatsyna19} showed that, for any Seidel matrix $S$, the coefficients of $\operatorname{Char}_S(x)$, which are all integers, must lie in certain congruence classes modulo powers of $2$. 
These necessary conditions have been used in \cite{GSY21,GSY22} to establish the nonexistence of Seidel matrices having certain prescribed spectra, which led to the best known (and best possible) upper bounds for equiangular line systems in certain dimensions.
The results of this paper can be thought of as a first step towards the goal of answering Question~\ref{que:main}.
Further steps, analogous to those of \cite{GSY21,GSY22}, may well be required for a complete answer to Question~\ref{que:main} but this is not attempted in this article.

In addition to ETFs and equiangular lines, the results of this paper can be applied to the spectral study of Hermitian Butson-Hadamard matrices~\cite{IMHermBut}, complete gain graphs~\cite{ReffGain12}, and Hermitian adjacency matrices of mixed graphs~\cite{Guo17}.

Throughout, we will require the use of some elementary results about cyclotomic fields for which we refer to Washington's textbook~\cite{Washington97}. 
In an attempt to make this paper self-contained, we present the necessary background about cyclotomic fields in Section~\ref{sec:prelims}.
Let $H \in \mathcal H_n(q)$.
Since $H$ is Hermitian, all of its eigenvalues are real and thus so are all the coefficients of $\operatorname{Char}_H(x)$.
Furthermore, since the entries of $H$ are algebraic integers from the cyclotomic field $\mathbb Q[\zeta]$, so are all the coefficients of $\operatorname{Char}_H(x)$.
Together, this means that the coefficients of $\operatorname{Char}_H(x)$ are real algebraic integers from $\mathbb Q[\zeta]$.
In other words, we have $\operatorname{Char}_H(x) \in \mathbb Z[\zeta+\zeta^{-1}][x]$ (see \cite[Proposition 2.16]{Washington97}).

Define $\rho := (1-\zeta)(1-\zeta^{-1})$.
We denote by $\mathcal X_n(q,e)$ the set of congruence classes of $\operatorname{Char}_H(x)$ modulo the principal ideal $\rho^e\mathbb Z[\zeta+\zeta^{-1}][x]$ where $H \in \mathcal H_n(q)$.
Note that when $q=2$, we have $\rho^e\mathbb Z[\zeta+\zeta^{-1}][x] = 2^{2e} \mathbb Z[x]$.
Our first main result corresponds to restrictions on the coefficients of the characteristic polynomials of matrices in $\mathcal H_n(q)$, which results in upper bounds for the cardinality of $\mathcal X_n(q,e)$.

\begin{theorem}
    \label{thm:mainBounds}
    Let $n, e, f \in \mathbb N$.
    \begin{enumerate}
        \item[(a)] If $n$ is even then $|\mathcal X_n(2,e/2)| \leqslant 2^{\lfloor e^2/2\rfloor -e+1}$.

        \item[(b)] If $n$ is odd then $|\mathcal X_n(2,e/2)| \leqslant 2^{\lceil e^2/2\rceil -e+1}$. 
        \item[(c)] If $p$ is an odd prime then $|\mathcal X_n(p^f,e)| \leqslant p^{(e-1)^2}$.
        \item[(d)] If $f > 1$ and $n$ is even then $|\mathcal X_n(2^f,e)| \leqslant 2^{(e-1)(e-2)+\lceil 2^{2-f}e \rceil - 1}$.
        \item[(e)] If  $f > 1$ and $n$ is odd then $|\mathcal X_n(2^f,e)| \leqslant 2^{(e-1)^2+\lceil 2^{2-f}e \rceil - 1-\lfloor e/4\rfloor}$.
    \end{enumerate}
\end{theorem}

Note that, for a non-prime-power $q$ and any $e \in \mathbb N$, the ideal $\rho^e \mathbb Z[\zeta+\zeta^{-1}]$ is equal to $\mathbb Z[\zeta+\zeta^{-1}]$.
Therefore, $\mathcal X_n(q,e)$ consists of just one element when $q$ is not a prime power and, furthermore, every polynomial in $\mathbb Z[\zeta+\zeta^{-1}][x]$ of degree $n$ belongs to the sole congruence class of $\mathcal X_n(q,e)$.
In contrast, for $q>2$ a power of a prime $p$, the ideal $\rho \mathbb Z[\zeta+\zeta^{-1}]$ is a prime ideal (see Section~\ref{sec:prelims}) of $\mathbb Z[\zeta+\zeta^{-1}]$ and, for any $n, e \in \mathbb N$, a large proportion of polynomials in $\mathbb Z[\zeta+\zeta^{-1}][x]$ of degree $n$ do not belong to any congruence class in $\mathcal X_n(q,e)$.
Loosely speaking, this means that the coefficients of the characteristic polynomial of $H \in \mathcal H_n(q)$ are more restricted when $q$ is a prime power.

Denote by $\mathcal X^\prime_n(q,e)$ the set of congruence classes of $\operatorname{Char}_S(x)$ modulo the principal ideal $\rho^e\mathbb Z[\zeta+\zeta^{-1}][x]$ where $S$ is a $q$-Seidel matrix of order $n$.
Note that bounds for $\mathcal X^\prime_n(2,e/2)$ were obtained by Greaves and Yatsyna~\cite{GreavesYatsyna19}.
Our second main result corresponds to restrictions on the coefficients of the characteristic polynomials of $q$-Seidel matrices, which results in upper bounds for the cardinality of $\mathcal X^\prime_n(q,e)$.

\begin{theorem}
    \label{thm:mainBounds2}
    Let $n, e, f \in \mathbb N$.
    \begin{enumerate}
        \item[(a)] If $n$ is even then $|\mathcal X^\prime_n(2,e/2)| \leqslant 2^{(e-2)(e-3)/2}$.
        \item[(b)] If $n$ is odd then $|\mathcal X^\prime_n(2,e/2)| \leqslant 2^{(e^2-5e+8)/2}$.
        \item[(c)] If $p$ is an odd prime then $|\mathcal X^\prime_n(p^f,e)| \leqslant p^{(e-1)^2}$.
        \item[(d)] If $f>1$ and $n$ is even then $|\mathcal X^\prime_n(2^f,e)| \leqslant 2^{(e-1)(e-2)}$.
        \item[(e)] If $f>1$ and $n$ is odd then $|\mathcal X^\prime_n(2^f,e)| \leqslant 2^{(e-1)^2-\lfloor e/4\rfloor}$.
    \end{enumerate}
\end{theorem}

It is interesting to observe that, in most cases, the bounds for $|\mathcal X_n(p^f,e)|$ and $|\mathcal X^\prime_n(p^f,e)|$ in Theorems~\ref{thm:mainBounds} and \ref{thm:mainBounds2} depend only on the prime $p$ and the exponent $e$ but not on the exponent $f$.
There are exceptions with part (d) and (e) in Theorem~\ref{thm:mainBounds} when $p = 2$.
For either parity of $n$, there are two different behaviours for the upper bounds of $|\mathcal X_n(p^f,e)|$ and $|\mathcal X^\prime_{n}(2^f,e)|$, depending on if $f=1$ or $f > 1$. 

By conducting computer experiments for small values ($n = 21$ and $e = 3,4$) of $n$ and $e$, we have empirical evidence that suggests the bounds in Theorem~\ref{thm:mainBounds} and Theorem~\ref{thm:mainBounds2} are sharp for fixed $q$ and $e \geqslant 3$, when $n$ is large enough compared to $e$.
We conjecture this to be the case.
\begin{conjecture}
For $e \in \mathbb N$ with $e \geqslant 3$ and $q$ a prime power, there exists $N \in \mathbb N$ such that for all $n \geqslant N$ we have equality in the corresponding bounds in Theorem~\ref{thm:mainBounds} and Theorem~\ref{thm:mainBounds2}.
\end{conjecture}

In order to prove our main results, we establish a generalisation of a classical result of Harary and Schwenk~\cite[Corollary 5a]{hs79}, which was rediscovered recently in \cite[Lemma 2.2]{GreavesYatsyna19}. 
For the convenience of the reader, we state this result here.
Throughout, we use $\phi$ to denote Euler's totient function and $\mathbf 1$ to denote the all-ones (column) vector. 

\begin{theorem}
\label{thm:hs}
    Let $\Gamma$ be a graph with adjacency matrix $A$ and
let $N \geqslant 3$ be an integer. Then
\begin{itemize}
    \item $\displaystyle \sum_{d | N} \phi(N/d) \operatorname{tr}(A^d) \equiv 0 \pmod {2N} \text{ if $N$ is odd}$;
    \item $\displaystyle \sum_{d | N} \phi(N/d) \operatorname{tr}(A^d) + \frac{N}{2} \mathbf 1^\top A \mathbf 1 \equiv 0 \pmod {2N} \text{ if $N$ is even.}$
\end{itemize}
\end{theorem}

The paper is organised as follows.
In Section~\ref{sec:prelims}, we recall some preliminary material about cyclotomic fields that will be required in the later sections of the paper.
In Section~\ref{sec:real}, we consider the real case, i.e., when $q = 2$ and prove parts (a) and (b) of our main theorems.
In Section~\ref{sec:graphs}, we introduce our graph-theoretic tools and generalise Theorem~\ref{thm:hs}.
In Section~\ref{sec:dets}, we produce results about the determinant of Hermitian matrices of roots of unity and provide proofs of parts (c) and (d) of our main theorems.
Finally, in Section~\ref{sec:euler}, we introduce the residue graph of a matrix in $\mathcal H_n(q)$ and, using so-called Euler graphs, we prove part (e) of our main theorems.

\section{Algebraic preliminaries}
\label{sec:prelims}

In this section, we develop some basic theory of ideals of cyclotomic fields.
A complex number $\alpha \in \mathbb C$ is called an \textbf{algebraic integer} if $\alpha$ is a zero of some monic integer polynomial.
Let $\zeta$ be a primitive $q$th root of unity.
The field extension $\mathbb Q[\zeta]$ of $\mathbb Q$ generated by $\zeta$ is called a \textbf{cyclotomic field}.
The set of algebraic integers in $\mathbb Q[\zeta]$ forms a subring called \textbf{ring of integers} of $\mathbb Q[\zeta]$.

First we list some basic facts that can be found in most introductory textbooks on algebraic number theory, for example \cite[Chapters 1 and 2]{Washington97}.

\begin{proposition}
\label{pro:idealprelim}
    Let $q > 2$ be a natural number and let $\zeta$ be a primitive $q$th root of unity.
    \begin{itemize}
        \item[(a)] $\mathbb Z[\zeta]$ is the ring of integers of the cyclotomic field $\mathbb Q[\zeta]$;
        \item[(b)] $\mathbb Z[\zeta+\zeta^{-1}]$ is the ring of integers of the field $\mathbb Q[\zeta] \cap \mathbb R$;
        \item[(c)] $\{1,\zeta+\zeta^{-1},\dots,(\zeta+\zeta^{-1})^{\phi(q)/2}\}$ is a basis for $\mathbb Z[\zeta+\zeta^{-1}]$ as a $\mathbb Z$-module;
        \item[(d)] $\{1,\zeta\}$ is a basis for $\mathbb Z[\zeta]$ as a $\mathbb Z[\zeta+\zeta^{-1}]$-module;
        \item[(e)] If $q$ is a prime power then $(1-\zeta)\mathbb Z[\zeta]$ is a prime ideal of $\mathbb Z[\zeta]$ otherwise $(1-\zeta)\mathbb Z[\zeta] = \mathbb Z[\zeta]$;
        \item[(f)] If $q$ is a power of a prime $p$ then $p \mathbb Z[\zeta] = (1-\zeta)^{\phi(q)}\mathbb Z[\zeta]$ and $((1-\zeta)\mathbb Z[\zeta])\cap \mathbb Z = p\mathbb Z$;
        \item[(g)] If $q$ is a power of $2$ then for any $\alpha, \beta \not \in (1-\zeta)\mathbb Z[\zeta]$ we have $\alpha + \beta \in (1-\zeta)\mathbb Z[\zeta]$.
    \end{itemize}
\end{proposition}


Next we prove a proposition about the maximum real subset of the ideal $(1-\zeta)\mathbb Z[\zeta]$.

\begin{proposition}
\label{pro:realIdeal}
Let $q>2$ be a power of a prime $p$ and let $\zeta$ be a primitive $q$th root of unity.
Then

$$((1-\zeta)\mathbb Z[\zeta])\cap \mathbb R = \rho \mathbb Z[\zeta+\zeta^{-1}].$$
\end{proposition}
\begin{proof}
    Let $\alpha \in ((1-\zeta)\mathbb Z[\zeta])\cap \mathbb R$.
    Then, by Proposition~\ref{pro:idealprelim} (b), $\alpha$ must be in $\mathbb Z[\zeta+\zeta^{-1}]$, since it is an algebraic integer in $\mathbb Q[\zeta] \cap \mathbb R$.
    Next observe that $\mathbb Z[\zeta+\zeta^{-1}] = \mathbb Z[\rho]$.
    Hence, we can write $\alpha$ as an integer linear combination of powers of $\rho$:
    \[
    \alpha = \sum_{i=0}^{\phi(q)/2-1}a_i \rho^i.
    \]
    Since $(1-\zeta)\mathbb Z[\zeta]$ is a prime ideal, we must have $a_0 \in (1-\zeta)\mathbb Z[\zeta]$.
    The proposition follows since $((1-\zeta)\mathbb Z[\zeta])\cap \mathbb Z = p\mathbb Z \subset \rho \mathbb Z[\zeta+\zeta^{-1}]$.
\end{proof}

Now we prove a key isomorphism of quotient rings.

\begin{proposition}
\label{pro:Idealcor}
Let $q>2$ be a power of a prime $p$ and let $\zeta$ be a primitive $q$th root of unity.
Then
\[
\frac{\mathbb Z[\zeta]}{(1-\zeta)\mathbb Z[\zeta]} \cong \frac{\mathbb Z[\zeta+\zeta^{-1}]}{\rho \mathbb Z[\zeta+\zeta^{-1}]} \cong \frac{\mathbb Z}{p\mathbb Z}.
\]
\end{proposition}
\begin{proof}
The isomorphism $\frac{\mathbb Z[\zeta]}{(1-\zeta)\mathbb Z[\zeta]} \cong \frac{\mathbb Z}{p\mathbb Z}$ is standard. 
Thus we only prove the left isomorphism.

First note that $\mathbb Z[\zeta] = \mathbb Z[\zeta+\zeta^{-1}] + \zeta \mathbb Z[\zeta+\zeta^{-1}]$ (from Proposition~\ref{pro:idealprelim} (d)).
Define the ring homomorphism $\tau : \mathbb Z[\zeta] \to \mathbb Z[\zeta+\zeta^{-1}]/\rho \mathbb Z[\zeta+\zeta^{-1}]$ by
$a + b\zeta \mapsto a + b + \rho \mathbb Z[\zeta+\zeta^{-1}]$, where $a$ and $b$ are in $\mathbb Z[\zeta+\zeta^{-1}]$.
It is straightforward to check that $\tau$ is indeed a ring homomorphism.

If $a + b\zeta \in (1-\zeta)\mathbb Z[\zeta]$ then $\tau(a + b\zeta) = \rho \mathbb Z[\zeta+\zeta^{-1}]$.
Indeed, we have $a = (1-\zeta)\alpha - b\zeta$ for some $\alpha \in \mathbb Z[\zeta]$.
Hence, 
$$\tau(a + b\zeta)=a+b + \rho \mathbb Z[\zeta+\zeta^{-1}]=(1-\zeta)(\alpha+b)+\rho \mathbb Z[\zeta+\zeta^{-1}] = \rho \mathbb Z[\zeta+\zeta^{-1}],$$
by Proposition~\ref{pro:realIdeal}.
Thus $(1-\zeta)\mathbb Z[\zeta] \subset \ker \tau$.
On the other hand, if $a+b \in \rho \mathbb Z[\zeta+\zeta^{-1}]$ then $a + b\zeta \in (1-\zeta)\mathbb Z[\zeta]$.
Indeed, we have $b = \rho\alpha - a$ for some $\alpha \in \mathbb Z[\zeta+\zeta^{-1}]$.
Hence, $a + b\zeta = (1-\zeta)a+\rho \alpha \zeta \in (1-\zeta)\mathbb Z[\zeta]$.
Thus $\ker \tau \subset (1-\zeta)\mathbb Z[\zeta]$.

Now $\frac{\mathbb Z[\zeta]}{(1-\zeta)\mathbb Z[\zeta]} \cong \frac{\mathbb Z[\zeta+\zeta^{-1}]}{\rho \mathbb Z[\zeta+\zeta^{-1}]}$ follows from the first isomorphism theorem for rings.
\end{proof}

It follows from Proposition~\ref{pro:Idealcor} that, for $q > 2$, the ideal $\rho\mathbb Z[\zeta+\zeta^{-1}]$ is a prime ideal of $\mathbb Z [\zeta+\zeta^{-1}]$.

We conclude this section with a couple of remarks about counting residues classes.

\begin{remark}
\label{rem:countResidues}
    Let $q=p^e$ for some prime $p$ and $e \in \mathbb N$ and let $k \in \mathbb N$.
    Suppose $a \in \rho^{k}\mathbb Z[\zeta+\zeta^{-1}]$.
    Since, by Proposition~\ref{pro:Idealcor}, the quotient ring $\mathbb Z/p\mathbb Z$ is isomorphic to $\mathbb Z[\zeta+\zeta^{-1}]/\rho \mathbb Z[\zeta + \zeta^{-1}]$,
    there are $p$ possible residues for $a$ modulo $\rho^{k+1}\mathbb Z[\zeta+\zeta^{-1}]$.
\end{remark}

Next we remark about a rational ideal inside the ideal $\rho^e \mathbb Z[\zeta+\zeta^{-1}]$.

\begin{remark}
\label{rem:rationalElements}
    Let $q=p^f$ be a prime power with
     $f > 1$.
     Let $g, h \in \mathbb N \cup \{0\}$ with $h < \phi(q)/2$.
     it is straightforward to obtain the equality
     \[
     (\rho^{\phi(q)g/2+h} \mathbb Z[\zeta+\zeta^{-1}]) \cap \mathbb Z = \begin{cases}
         p^{g + 1}\mathbb Z, & \text{if $h > 0$;} \\
         p^{g}\mathbb Z, & \text{if $h = 0$}.
     \end{cases}
     \]
     Indeed, suppose that $a \in \mathbb Z$ belongs to $\rho^{\phi(q)g/2+h} \mathbb Z[\zeta+\zeta^{-1}]$.
     Using Proposition~\ref{pro:idealprelim} (f), we have $a \in p^g\rho^{h} \mathbb Z[\zeta+\zeta^{-1}]$.
     If $h=0$ then we are done.
     Otherwise, if $h>0$ then $a = p^g b$ for some $b \in (1-\zeta)\mathbb Z[\zeta+\zeta^{-1}$.
     By Proposition~\ref{pro:idealprelim} (f), we have $p \in (1-\zeta)\mathbb Z$.
     Since $b$ is an integer, it follows that $b \in p\mathbb Z$, as required.
     
     Therefore, if $a \in \mathbb Z$ then there are $p^{\lceil 2e/\phi(q) \rceil}$ possible residues for $a$ modulo  $\rho^{e}\mathbb Z[\zeta+\zeta^{-1}]$.
     Furthermore, if $a \in p\mathbb Z$ then there are $p^{\lceil 2e/\phi(q) \rceil-1}$ possible residues for $a$ modulo  $\rho^{e}\mathbb Z[\zeta+\zeta^{-1}]$.
\end{remark}

In the remainder of the paper $\zeta$ is to be understood as a primitive $q$th root of unity.

\section{The real case}
\label{sec:real}

In this section, we consider matrices in $\mathcal H_n(2)$, that is, symmetric $\{ \pm 1\}$-matrices.
The proofs of part (a) and part (b) of Theorem~\ref{thm:mainBounds} and Theorem~\ref{thm:mainBounds2} are provided herein.

First we remark about the top three coefficients of the characteristic polynomial of a matrix in $\mathcal H_n(q)$.

\begin{remark}
     \label{rem:a0a1a2}
    Let $H \in \mathcal H_n(q)$.
    Observe that the trace of $H^2$ is equal to $n^2$.
    Write $\operatorname{Char}_{H}(x)=\sum_{i=0}^na_ix^{n-i}$.
    Using Newton's identities, we see that $a_0=1$  
    and $a_2=(a_1^2-n^2)/2$.
    Furthermore, the parity of $a_1$ is the same as that of $n$.
\end{remark}

Next we show the determinant of a matrix in $H_n(2)$ is divisible by a large power of $2$.

\begin{lemma}
    \label{lem:detIdealEasy}
    Let $n \in \mathbb N$ and let $H$ be a $\{\pm 1\}$-matrix of order $n$.
    Then $\det (H) \in 2^{n-1}\mathbb Z$.
\end{lemma}
\begin{proof}
    Subtract the first row of $H$ from all other rows to obtain $H^\prime$.
    Then, except for the first row, each entry of $H^\prime$ is in $2\mathbb Z$.
    Thus $\det (H) = \det (H^\prime) \in  2^{n-1}\mathbb Z$.
\end{proof}

Since the coefficients of the characteristic polynomial of a matrix are the sums of its principal minors of the appropriate size, we immediately have the following corollary of Lemma~\ref{lem:detIdealEasy}.

\begin{corollary}
    \label{cor:detIdealEasy}
    Let $n \in \mathbb N$, $k \in \{1,\dots,n\}$, $H \in \mathcal H_n(2)$, and write $\operatorname{Char}_H(x)=\sum_{i=0}^n a_ix^{n-i}$.
    Then $a_{k} \in 2^{k-1} \mathbb Z$.
\end{corollary}

We will make use of the following lemma from \cite{GreavesYatsyna19}.

\begin{lemma}[{\cite[Corollary 3.2 and Lemma 3.8]{GreavesYatsyna19}}]
\label{lem:GGyats}
    Let $A$ be the adjacency matrix of a graph of order $n$ and write $\operatorname{Char}_{J-2A}(x) = \sum_{i=0}^n a_ix^{n-i}$. 
    Then $a_k \in 2^k \mathbb Z$ for all $k$ even.
    Furthermore, if $n$ is even then $a_k \in 2^k \mathbb Z$ for all $k$.
\end{lemma}

Given a matrix $M$ of order $n$ and a subset $S \subseteq \{1,\dots,n\}$, we denote by $M[S]$, the principal submatrix of $M$ obtained by deleting the $i$th row and column of $M$ for each $i \in S$.
The $(i,j)$-entry of $M$ is denoted by $M_{i,j}$.

We will make use of the matrix determinant lemma (see, e.g., \cite{matdetHarville}), which we now state.

\begin{lemma}
    \label{lem:matdetl}
    Let $M$ be a square matrix and let $\mathbf u, \mathbf v$ be column vectors of the appropriate size.
    Then
    \[
    \det(M + \mathbf u \mathbf v^\top) = \det(M) + \mathbf v^\top \operatorname{adj}(M)\mathbf v.
    \]
\end{lemma}

Let $H \in \mathcal H_n(2)$.
We define $\Delta(H)$ as the set of indices $i$ such that $H_{i,i} = -1$.
Observe that we can write $H = J-2A-2\operatorname{diag}(\chi_{\Delta(H)})$, where $A$ is a graph-adjacency matrix and $\chi_{\Delta(H)}$ is the indicator vector for $\Delta(H)$.

\begin{lemma}
\label{lem:matdet1}
    Let $H \in \mathcal H_n(2)$ and write $H = J-2A-2\operatorname{diag}(\chi_{\Delta(H)})$.
    Then
    \[
    \operatorname{Char}_{H}(x) = \sum_{S \subseteq \Delta(H)} 2^{|S|}\operatorname{Char}_{J-2A[S]}(x).
    \]
\end{lemma}
\begin{proof}
Induction on the cardinality of $\Delta(H)$.
When $\Delta(H) = \emptyset$ the statement is obvious.
Otherwise, apply Lemma~\ref{lem:matdetl} with $M = xI-J+2A+2\operatorname{diag}(\chi_{\Delta(H)\backslash \{i\}})$, $\mathbf u = \chi_{\{i\}}$, and $\mathbf v = 2\mathbf u$ to obtain
\begin{align}
\label{eqn:matdet}
    \operatorname{Char}_{H}(x) &= \operatorname{Char}_{J-2A-2\operatorname{diag}(\chi_{\Delta(H)\backslash \{i\}})}(x)+ 2\operatorname{Char}_{J-2A[i]-2\operatorname{diag}(\chi_{\Delta(H)})[i]}(x).
\end{align}
By induction, we have 
\begin{align}
    \operatorname{Char}_{J-2A-2\operatorname{diag}(\chi_{\Delta(H)\backslash \{i\}})} &= \sum_{S \subseteq \Delta(H)\backslash \{i\}} 2^{|S|}\operatorname{Char}_{J-2A[S]}(x) \nonumber \\
    &= \sum_{\substack{S \subseteq \Delta(H) \\ i \not \in S}}2^{|S|}\operatorname{Char}_{J-2A[S]}(x); \label{eqn:ind1} \\
    2\operatorname{Char}_{J-2A[i]-2\operatorname{diag}(\chi_{\Delta(H)})[i]} &= \sum_{S \subseteq \Delta(H)\backslash \{i\}} 2^{|S|+1}\operatorname{Char}_{J-2A[S\cup \{i\}]}(x) \nonumber \\
    &= \sum_{\substack{S \subseteq \Delta(H) \\ i \in S}} 2^{|S|}\operatorname{Char}_{J-2A[S]}(x). \label{eqn:ind2}
\end{align}
Combine \eqref{eqn:matdet} with \eqref{eqn:ind1} and \eqref{eqn:ind2} to obtain the statement of the lemma.
\end{proof}

\begin{corollary}
    \label{cor:finalConds}
    Let $n \in \mathbb N$, $k \in \{1,\dots,n\}$, $H \in \mathcal H_n(2)$, and write $\operatorname{Char}_H(x)=\sum_{i=0}^n a_ix^{n-i}$.
    Then $a_{k} \in 2^{k-1} \mathbb Z$.
    Furthermore, if $n+k$ is odd then $a_{k} \in 2^{k} \mathbb Z$.
\end{corollary}
\begin{proof}
    By Lemma~\ref{lem:matdet1}, we can write
    \[
    a_k = \sum_{\substack{S \subseteq \Delta(H) \\ |S| \text{ even} \\ |S| \leqslant k}} 2^{|S|} c_{k-|S|}(S) + \sum_{\substack{S \subseteq \Delta(H) \\ |S| \text{ odd} \\ |S| \leqslant k}} 2^{|S|} c_{k-|S|}(S),
    \]
    where, for each $S \subseteq \Delta(H)$, we have $\operatorname{Char}_{J-2A[S]}(x) =\sum_{i=0}^{n-|S|} c_i(S) x^{n-|S|-i}$.
    
    By Lemma~\ref{lem:GGyats}, the coefficient $c_i(S)$ is in $2^i\mathbb Z$ if $i$ is even or is $n-|S|$ is even.
    Hence $2^{|S|}c_{k-|S|}(S)$ is in $2^k\mathbb Z$ if $k-|S|$ is even or $n-|S|$ is even.
    Now it readily follows that $a_k \in 2^k\mathbb Z$ when $n+k$ is odd.
\end{proof}

Now we are ready to prove the main result of this section.

\begin{proof}[Proof of Theorem~\ref{thm:mainBounds} (a) and (b)]
    Let $H \in \mathcal H_n(2)$ and write $\operatorname{Char}_H(x)=\sum_{i=0}^n a_ix^{n-i}$.
    By Remark~\ref{rem:a0a1a2}, there are at most $2^{e-1}$ possible residues for $a_1$ modulo $2^e$ and just one possible residue for $a_0$ and $a_2$ (since $a_2$ is determined by the value of $a_1$).
    By Corollary~\ref{cor:detIdealEasy}, there are at most $2^{e-i+1}$ possible residues for each $a_i$ modulo $2^e$ where $i \in \{3,\dots,e\}$ and just one possible residue for each $a_i$ modulo $2^e$ where $i \in \{e+1,\dots,n\}$.
    Furthermore, by Corollary~\ref{cor:finalConds}, for each $i \in \{3,\dots,e\}$ whose parity is different from that of $n$, there are at most $2^{e-i}$ possible residues for each $a_i$ modulo $2^e$.
    
    Thus, we have at most
    \[
    2^{e-1} \prod_{\substack{i \in \{3,\dots,e\} \\ n + i \in 2\mathbb Z}} 2^{e-i+1} \prod_{\substack{i \in \{3,\dots,e\} \\ n + i \not \in 2\mathbb Z}} 2^{e-i}
    \]
    possible residues for $\operatorname{Char}_H(x)$ modulo $2^e \mathbb Z[x]$, as required.
\end{proof}

Parts (a) and (b) of Theorem~\ref{thm:mainBounds2} were proved in \cite[Corollaries 3.4 and 3.13]{GreavesYatsyna19}.

\section{Underlying graphs, walks, and their weights}
\label{sec:graphs}

In this section, we introduce the underlying graph of a matrix in $\mathcal H_n(q)$, its walks, and their weights.
The main result (Corollary~\ref{cor:mainHSgen}) of this section is about a congruence of traces, which generalises a result of Harary and Schwenk~\cite[Corollary 5a]{hs79} from 1979 (see Theorem~\ref{thm:hs}).

Fix $n, q \in \mathbb N$ with $q \geqslant 2$ and let $H \in \mathcal H_n(q)$ and $A=(J-H)/(1-\zeta)$.
The \textbf{underlying graph} $\Gamma(H)$ of $H$ is defined to have vertex set $\{1,\dots,n\}$ where $i$ is adjacent to $j$ if and only if $A_{i,j} \ne 0$.
It is easy to see that, for all $i$ and $j$, we have 
$A_{i,j} \ne 0$ if and only if $A_{j,i}\ne 0$, 
thus the adjacency of $\Gamma(H)$ is well-defined.
Note that $-1$ is a $q$th root of unity when $q$ is a power of $2$.
Thus the underlying graph $\Gamma(H)$ may have loops when $q$ is a power of $2$, hence it may not be a simple graph, however, all graphs considered herein do not have multiedges.
Furthermore, observe that, when $q=2$ and $H_{i,i} = 1$ for all $i \in \{1,\dots,n\}$, the matrix $A$ is a symmetric $\{0,1\}$-matrix with zero diagonal.
In fact, $A$ is the adjacency matrix for the underlying graph of $H$.
In this sense, one can think of $A$ as a (non-Hermitian) weighted generalisation of an adjacency matrix for $q > 2$.

\begin{remark}
    \label{rem:Aentries}
    Note that $A_{i,j} + A_{j,i} \in (1-\zeta)\mathbb Z[\zeta]$.
    Indeed, suppose $A_{i,j} = (1-\zeta^e)/(1-\zeta)$ for some $e$.
    Then $A_{i,j} = (1-\zeta^{-e})/(1-\zeta)$.
    If $e = 0$ then our claim follows immediately.
    For each nonzero $e$, we have 
    \[
    \frac{1-\zeta^{e}}{1-\zeta} = 
    \begin{cases}
        1+\zeta + \dots + \zeta^{e-1} & \text{ if $e > 0$; }\\
        1+\zeta + \dots + \zeta^{q+e-1} & \text{ if $e < 0$. }
    \end{cases}
    \]
    Thus, using Proposition~\ref{pro:idealprelim} (f), we have $A_{i,j} - e \in (1-\zeta)\mathbb Z[\zeta]$.
    Furthermore, $A_{j,i} +e \in (1-\zeta)\mathbb Z[\zeta]$, whence our claim follows.
\end{remark}

Let $G$ be a group acting on a set $X$.
The orbit $\operatorname{orb}_G(x)$ of an element $x \in X$ is defined as the set $\{x^g \;|\; g \in G \}$.
Let $\Gamma = \Gamma(H)$ and let $\mathbf w$ be a walk of length $N$ in $\Gamma$; we write
$\mathbf w = (w_0,w_1,\dots,w_N)$ where each vertex $w_i$ is adjacent to the vertex $w_{i+1}$ for each $i \in \{0, \dots, N - 1\}$.
The \textbf{weight} $\operatorname{wt}(\mathbf w)$ of $\mathbf w$ is defined as
\[
    \operatorname{wt}(\mathbf w) = \prod_{i=0}^{N-1} A_{w_i,w_{i+1}}.
\]
\begin{remark}
    \label{rem:weightPowers}
    The sum of the $l$th powers of the weights of all $k$-walks from vertex $i$ to vertex $j$ in the graph $\Gamma(H)$ is given by the $(i,j)$-entry of $(A^{\circ l})^k$, where $A^{\circ l}$ denotes the Hadamard (or Schur) product of the matrix $A$ with itself $l$ times.
    This statement can be readily verified by induction on $k$.
\end{remark}

Let $\mathbf w = (w_0,w_1,\dots,w_N)$ be a closed $N$-walk of $\Gamma$, i.e., $w_0 = w_N$.
We say that $\mathbf w$ is \textbf{simple} if each pair of consecutive vertices is distinct, i.e., $w_i \ne w_{i+1}$ for all $i \in \{0,\dots,N-1\}$.
Denote by $W_N(\Gamma)$ the set of closed $N$-walks of $\Gamma$ and denote by $W_N^\prime(\Gamma)$ the set of \emph{simple} closed $N$-walks of $\Gamma$.

For $N \geqslant 3$, there is a natural correspondence between the vertices of the closed walk $\mathbf w$ and the vertices of a regular $N$-gon.
Under this correspondence, we consider the dihedral group $D_N$ of order $2N$ acting on the set of
closed $N$-walks of $\Gamma$. 
Let $N \geqslant 3$ and write $D_N = \langle r, s \; |\; r^N, s^2, (rs)^2 \rangle$.
For notational convenience, we denote by $\mathbf w^k:=(w_k,w_{k+1},\dots,w_{k+N})$ (with indices reduced modulo $N$) the image of $\mathbf w$ under the action of $r^k$ and we denote by $\mathbf w^{\top}:=(w_N,w_{N-1},\dots,w_{0})$, the image of $\mathbf w$ under the action of $s$.
Under the action of $r^i s$, the image of $\mathbf w$ is $(w_{N+i},w_{N+i-1},\dots,w_{i})$.

We denote by $\operatorname{fix}_{\Gamma}(g)$ the set of walks $\mathbf w \in W_N(\Gamma)$ such that $\mathbf w$ is equal to its image under the action of $g$.
For our purposes, the ambient group to which the element $g$ belongs will always be the dihedral group $D_N$.

The walk $\mathbf w$ is called \textbf{palindromic} if it is equal to its reverse walk $\mathbf w^\top$.
Note that if $\mathbf w \in W_N^\prime(\Gamma)$ is palindromic then $N$ must be even.

\begin{remark}
    \label{rem:nonsimple}
    We remark that if $q > 2$ is a power of $2$ and $H \in \mathcal H_n(q)$, then the diagonal entries of $A=(J-H)/(1-\zeta)$ are either equal to $0$ or to $2/(1-\zeta)$.
    By Proposition~\ref{pro:idealprelim} (f), both such entries belong to the ideal $(1-\zeta)\mathbb Z[\zeta]$.
    Therefore, the weight of any non-simple walk of $\Gamma(H)$ is also in $(1-\zeta)\mathbb Z[\zeta]$.
    On the other hand, if $q$ is an odd prime power and $H \in \mathcal H_n(q)$ then $\Gamma(H)$ is a simple graph.
\end{remark}

\begin{lemma}
    \label{lem:kappa}
    Let $\Gamma$ be a graph, $N \geqslant 3$ be an integer, and $\mathbf w \in W_N^\prime(\Gamma)$ be palindromic.
    Suppose there exists $k \in \mathbb N \cup \{0\}$ such that $N \in 2^k\mathbb Z$, and $\mathbf w^{N/2^i} = \mathbf w$ for all $i \in \{0,\dots,k\}$.
    Then $N \in 2^{k+1}\mathbb Z$.
\end{lemma}
\begin{proof}
    Let $\mathbf w = (w_0,w_1,\dots,w_{N-1},w_N)$.
    Since $\mathbf w$ is palindromic, for each $i \in \{0,\dots,N\}$, we have $w_i = w_{N-i}$.
    If $N$ were odd then, for $i = (N-1)/2$, we would have $w_i = w_{i+1}$, which is impossible since $\mathbf w$ is simple.
    Thus $N$ must be even.
    
    Now suppose that $\mathbf w^{N/2} = \mathbf w$.
    This means that, for each $i \in \{0,\dots,N/2\}$, we have $w_i = w_{N/2+i}$.
    Combining this with the identity $w_i = w_{N-i}$ yields $w_i = w_{N/2-i}$, for each $i \in \{0,\dots,N/2\}$.
    Similar to the above, since $\mathbf w$ is simple, $N/2$ must be even.
    
    The proof is completed by continuing this argument inductively.
\end{proof}

Let $\mathbf w = (w_0,w_1,\dots,w_{N-1},w_N)$ be a palindromic walk.
In view of Lemma~\ref{lem:kappa}, we can define the parameter $\kappa(\mathbf w):= \min \{ k \in \mathbb N \; | \; \mathbf w^{N/2^k} \ne \mathbf w \}$ and set $\psi(\mathbf w) := N/2^{\kappa(\mathbf w)}$.
Define the map $\Psi: W_N^\prime(\Gamma) \to W_N^\prime(\Gamma)$ as
\[
    \Psi(\mathbf w) = 
    \begin{cases}
        \mathbf w^{\psi(\mathbf w)}, & \text{ if $\mathbf w$ is palindromic;} \\
        \mathbf w^{\top}, & \text{ otherwise.}
    \end{cases}
\]
Observe that, for all $\mathbf w \in W_N^\prime(\Gamma)$, we have $\operatorname{wt}(\Psi(\mathbf w)) = \operatorname{wt}(\mathbf w^{\top})$.

\begin{lemma}
    \label{lem:evenPalin}
    Let $\Gamma$ be a graph, $N \geqslant 3$ be an integer, and $\mathbf w \in W_N^\prime(\Gamma)$.
    Then number of palindromic walks in $\operatorname{orb}_{D_N}(\mathbf w)$ is either $0$ or $2$.
\end{lemma}
\begin{proof}
    If $\mathbf w^k \neq (\mathbf w^k)^\top $ for all natural numbers $k$, then $ \operatorname{orb}_{D_N}(\mathbf w)$ does not contain any palindromic walk.
    Otherwise, for some $k$, we have $\mathbf w^k$ is a palindromic walk.
    Since $\mathbf w$ is simple and palindromic, we know that $N$ is even.
    Note that $ \operatorname{orb}_{D_N}(\mathbf w^k)= \operatorname{orb}_{D_N}(\mathbf w)$.
    Without loss of generality, we can assume $\mathbf w$ is palindromic and write 
    \begin{align*}
    \mathbf w
    &=(w_0,w_{1},\dots ,w_{\frac{N}{2}-1},w_{\frac{N}{2}}, w_{\frac{N}{2}-1},\dots ,w_{1},w_0).
    \end{align*}
    For $i \in \mathbb N \cup \{ 0 \}$ such that $2^i$ divides $N$, define $\mathbf x_i :=(w_0,w_{1},\dots ,w_{N/2^i-1},w_{N/2^i})$.
    Define the concatenation of $\mathbf x_i \mathbf x_i^\top$ as follows
    $$\mathbf x_i \mathbf x_i^\top := (w_0,w_{1},\dots ,w_{N/2^i-1},w_{N/2^i}, w_{N/2^i-1},\dots ,w_{1},w_0).$$
    Set $\kappa = \kappa(\mathbf w)$.
    Then
        $$\mathbf w=\underbrace{\mathbf x_\kappa \mathbf x_\kappa^\top \dots \mathbf x_\kappa \mathbf x_\kappa^\top}_{\substack{2^{\kappa-1}}} \quad \text{ and } \quad \mathbf w^{\psi(\mathbf w)}=\underbrace{\mathbf x_\kappa^\top \mathbf x_\kappa \dots \mathbf x_\kappa^\top \mathbf x_\kappa}_{\substack{2^{\kappa-1}}}$$
    are two distinct palindromic walks in $\operatorname{orb}_{D_N} (\mathbf w)$.

    It remains to show that there is no other $m \in \{1,\dots,N\}$ such that $\mathbf w^m$ is palindromic.
    It suffices to show that there is no $m \in \{1,\dots,N/2^\kappa\}$ such that $(\mathbf x_\kappa \mathbf x_\kappa^\top )^m = ((\mathbf x_\kappa \mathbf x_\kappa^\top )^m)^\top$.
    Since $\mathbf x_\kappa \neq \mathbf x_\kappa^\top$, there exists an integer $l \in \{0,\dots,N/2^\kappa \}$ such that the $l$th position of $\mathbf x_\kappa$ and $\mathbf x_\kappa^\top$ differ, i.e., $w_l \neq w_{N/2^\kappa - l}$.
    It follows that, for each $m \in \{1,\dots,N/2^\kappa\}$, the $(l+m)$-th position of  $(\mathbf x_\kappa \mathbf x_\kappa^\top )^m$ and $((\mathbf x_\kappa \mathbf x_\kappa^\top )^m)^\top$ also differ, whence
    $$(\mathbf x_\kappa \mathbf x_\kappa^\top )^m \neq ((\mathbf x_\kappa \mathbf x_\kappa^\top )^m)^\top.$$
    Thus, $\mathbf w$ and $\mathbf w^{\psi (\mathbf w)}$ are the only two palindromic walks in $\operatorname{orb}_{D_N} (\mathbf w)$.
\end{proof}

Using Lemma~\ref{lem:evenPalin}, we can partition the set of orbits of a simple walk in a graph into two equal-sized parts in a natural way.
We use $\sqcup$ to denote the disjoint union operator.

\begin{lemma}
    \label{lem:Upart}
    Let $\Gamma$ be a graph, $N \geqslant 3$ be an integer, and $\mathbf w \in W_N^\prime(\Gamma)$.
    Then there exists $U \subset \operatorname{orb}_{D_N}(\mathbf w)$ such that
    \[
        \operatorname{orb}_{D_N}(\mathbf w) = U \; \sqcup \;  \Psi(U).
    \]
\end{lemma}
\begin{proof}
By Lemma~\ref{lem:evenPalin}, the number of palindromic walks in $\operatorname{orb}_{D_N}(\mathbf w)$ is either $0$ or $2$.
If it is $0$ then set $U = \operatorname{orb}_{\langle r \rangle}(\mathbf w)$; clearly $U$ and $\Psi(U)$ partition $\operatorname{orb}_{D_N}(\mathbf w)$.
Otherwise, if it is $2$ then, without loss of generality, assume that $\mathbf w$ is palindromic and set $U = \{\mathbf w^i \; | \; i \in \{0,\dots,\psi(\mathbf w)-1\} \}$.
    Using the proof of Lemma~\ref{lem:evenPalin}, it follows that $U$ and $\Psi(U)$ partition $\operatorname{orb}_{D_N}(\mathbf w)$.
\end{proof}

\begin{remark}
    \label{rem:sameWeight}
    Let $H \in \mathcal H_n(q)$, $\Gamma=\Gamma(H)$, $N \geqslant 3$ be an integer, $\mathbf w \in W_N^\prime(\Gamma)$, and let $U$ be a subset from Lemma~\ref{lem:Upart} such that $ \operatorname{orb}_{D_N}(\mathbf w) = U \; \sqcup \;  \Psi(U)$.
    Observe that each walk $\mathbf w \in U$ has the same weight.
    It is important to note, however, that it is not necessary for $\mathbf w$ and $\mathbf w^\top$ to have the same weight.
\end{remark}

Using Lemma~\ref{lem:Upart}, we can also partition the set $W_k^\prime(\Gamma)$ of closed $k$-walks of a graph $\Gamma$ into two equal-sized parts in the same way.

\begin{corollary}
        \label{cor:UpartWalk}
    Let $\Gamma$ be a graph, $k \geqslant 1$ be an integer, and $\mathbf w \in W_k^\prime(\Gamma)$.
    There exists $U \subset W_k^\prime(\Gamma)$ such that
    \[
        W_k^\prime(\Gamma) = U \; \sqcup \;  \Psi(U).
    \]
\end{corollary}
\begin{proof}
     The case when $k=1$ is vacuous.
     When $k=2$, each walk $(v,w,v)$ can be paired up with $(w,v,w)$ and each non-palindromic walk $(u,v,w)$ can be paired up with $(w,v,u)$.
     For $k \geqslant 3$, the corollary follows from Lemma~\ref{lem:Upart}.
\end{proof}

We will also need the following result about the weights of (not necessarily simple) walks in an orbit. 

\begin{lemma}
    \label{lem:UpartNonsimple}
     Let $H \in \mathcal H_n(q)$, $\Gamma=\Gamma(H)$, $N \geqslant 3$ be an integer, and $\mathbf w \in W_N(\Gamma)$.
    Then either
    \begin{itemize}
        \item $\operatorname{wt}(\mathbf x) = \operatorname{wt}(\mathbf w)$ for all $\mathbf x \in \operatorname{orb}_{D_N}(\mathbf w)$; or
        \item there exists $U \subset \operatorname{orb}_{D_N}(\mathbf w)$ such that $ \operatorname{orb}_{D_N}(\mathbf w) = U \; \sqcup \;  \{ \mathbf x^\top \; | \; \mathbf x \in U\}$
    and $\operatorname{wt}(\mathbf x) = \operatorname{wt}(\mathbf w)$ for all $\mathbf x \in U$.
    \end{itemize}
\end{lemma}
\begin{proof}
Suppose $\mathbf x \in \operatorname{orb}_{D_N}(\mathbf w)$ such that $\operatorname{wt}(\mathbf x) \ne \operatorname{wt}(\mathbf x^\top)$.
Without loss of generality, we can assume that $\operatorname{wt}(\mathbf x) = \operatorname{wt}(\mathbf w)$.
Let $U = \{ \mathbf y \in \operatorname{orb}_{D_N}(\mathbf w) \; | \; \mathbf y^i = \mathbf x \text{ for some } i \in \{0,\dots, N-1\} \}$.
Clearly, $\operatorname{wt}(\mathbf y) = \operatorname{wt}(\mathbf w)$ for each $\mathbf y \in U$ and $\operatorname{orb}_{D_N}(\mathbf w) \backslash U =  \{ \mathbf y^\top \; | \; \mathbf y \in U\}$, as required.
\end{proof}

In the subsequent results in this section, we will occasionally take advantage of the following properties of weights of walks when $q$ is a power of $2$, which we record as a remark.

\begin{remark}
    \label{rem:weightsum}
    Let $n,k \in \mathbb N$, $q \geqslant 2$ be a power of $2$, $H \in \mathcal H_n(q)$, $\Gamma =\Gamma(H)$, and $\mathbf w \in W_k(\Gamma)$.
    We claim that $\operatorname{wt}(\mathbf w ) + \operatorname{wt}(\Psi(\mathbf w) ) \in (1-\zeta)\mathbb Z[\zeta]$.
    Indeed, write $\mathbf w = (w_0,w_1,\dots,w_k)$ and $\operatorname{wt}(\mathbf w) = \prod_{i=1}^{k} \frac{1-\zeta^{e_i}}{1-\zeta}$ for some nonzero integers $e_i$.
    Then $\operatorname{wt}(\Psi(\mathbf w) ) = \prod_{i=1}^{k} \frac{1-\zeta^{q-e_i}}{1-\zeta}$.
    Using Remark~\ref{rem:Aentries} and Proposition~\ref{pro:idealprelim} (f), we have 
    \[
    \operatorname{wt}(\mathbf w) - \prod_{i=1}^{k} e_i \in (1-\zeta)\mathbb Z[\zeta] \text{ and } \operatorname{wt}(\Psi(\mathbf w)) - \prod_{i=1}^{k} (-e_i) \in (1-\zeta)\mathbb Z[\zeta].
    \]
    Our claim now readily follows from Proposition~\ref{pro:idealprelim} (g).
\end{remark}

Next, we apply the orbit-stabiliser theorem to obtain a key result about a certain sum of weights of walks.

\begin{lemma}
\label{lem:weightedBurnside}
    Let $n \in \mathbb N$, $q > 2$ be a power of $2$, $N \geqslant 3$ be an integer, $H \in \mathcal H_n(q)$, and $\Gamma=\Gamma(H)$.
    Then
$$\sum_{g \in D_N} \sum_{\mathbf w \in \operatorname{fix}_{\Gamma}(g)} \operatorname{wt}(\mathbf w) \in (1-\zeta)N\mathbb Z[\zeta].$$
\end{lemma}
\begin{proof}
For each $\mathbf w \in W_N(\Gamma)$, denote by $\operatorname{stab}_{D_N}(\mathbf w)$ the stabiliser of $\mathbf w$ under the action of $D_N$ and the set of orbits is denoted by $\mathcal O = \{\operatorname{orb}_{D_N}(\mathbf w) \; | \; \mathbf w \in W_N(\Gamma) \}$.
Using the orbit-stabiliser theorem, we can write
    \begin{align*}
        \sum_{g \in D_N} \sum_{\mathbf w \in \operatorname{fix}_\Gamma(g)} \operatorname{wt}(\mathbf w) 
        &= \sum_{\mathbf w \in W_N(\Gamma)} \sum_{g \in \operatorname{stab}_{D_N}(\mathbf w)} \operatorname{wt}(\mathbf w) 
        \\
        &= \sum_{\mathbf w \in W_N(\Gamma)}  \frac{2N \operatorname{wt}(\mathbf w)}{|\operatorname{orb}_{D_N}(\mathbf w)|}
         \\
        &= \sum_{X \in \mathcal O}\sum_{\mathbf w \in X}  \frac{2N \operatorname{wt}(\mathbf w)}{|X|}.
    \end{align*}
    By Lemma~\ref{lem:UpartNonsimple}, for each orbit $X \in \mathcal O$, the number of distinct weights of walks in $X$ is at most $2$. 
    Now we partition $\mathcal O$ as  $\mathcal O = \mathcal O_0 \sqcup \mathcal O_1 \sqcup \mathcal O_2$, where $\mathcal O_0$ is the set of orbits consisting of simple walks, $\mathcal O_1$ is the set of orbits that contain non-simple walks with just one weight, and $\mathcal O_2$ is the set of orbits that contain non-simple walks with two distinct weights.
    We consider each part separately.
    
    First, we consider the orbits that consist of simple walks.
    By Lemma~\ref{lem:Upart}, for each $X \in \mathcal O_0$, there exists a subset $U(X) \subset X$ such that
        \[
       X = U(X) \; \sqcup \;  \Psi(U(X)),
    \]
    and, by Remark~\ref{rem:sameWeight}, each walk in $U(X)$ has the same weight.
    Define $\operatorname{wt}(X) := \operatorname{wt}(\mathbf w)+\operatorname{wt}(\mathbf w^\top)$ for some $\mathbf w \in X$.
    Then
    \begin{align*}
        \sum_{X \in \mathcal O_0}\sum_{\mathbf w \in X}  \frac{2N \operatorname{wt}(\mathbf w)}{|X|} &= 
         \sum_{X \in \mathcal O_0}\sum_{\mathbf w \in U(X)}  N\frac{ \operatorname{wt}(\mathbf w)+\operatorname{wt}(\mathbf w^\top)}{|U(X)|} = \sum_{X \in \mathcal O_0} \operatorname{wt}(X) N.
    \end{align*}
    By Remark~\ref{rem:weightsum}, we have $\operatorname{wt}(X) \in (1-\zeta)\mathbb Z[\zeta]$ for all $X \in \mathcal O_0$, which implies that $\sum_{X \in \mathcal O_0} \operatorname{wt}(X) N \in (1-\zeta)N\mathbb Z[\zeta]$.
    
    Next we consider the orbits that contain non-simple walks and have just one distinct weight.
    By Lemma~\ref{lem:UpartNonsimple}, for each $X \in \mathcal O_1$, every walk in $X$ has the same weight.
    Define $\operatorname{wt}(X) := \operatorname{wt}(\mathbf w)$ for some $\mathbf w \in X$.
    Since $q > 2$ is a power of $2$, by Remark~\ref{rem:nonsimple}, we have $\operatorname{wt}(X) \in (1-\zeta)\mathbb Z[\zeta]$ for each $X \in \mathcal O_1$.
    Therefore 
        \begin{align*}
        \sum_{X \in \mathcal O_1}\sum_{\mathbf w \in X}  \frac{2N \operatorname{wt}(\mathbf w)}{|X|} &=  \sum_{X \in \mathcal O_1} 2\operatorname{wt}(X) N \in (1-\zeta)N\mathbb Z[\zeta].
    \end{align*}
    
    Finally, we consider the orbits that contain non-simple walks and have two distinct weights.
    By Lemma~\ref{lem:UpartNonsimple}, for each $X \in \mathcal O_2$, there exists a subset $U(X) \subset X$ such that
        \[
       X = U(X) \; \sqcup \;  \{ \mathbf x^\top \; |\; \mathbf x \in U(X)\},
    \]
    and each walk in $U$ has the same weight.
    Define $\operatorname{wt}(X) := \operatorname{wt}(\mathbf w)+\operatorname{wt}(\mathbf w^\top)$ for some $\mathbf w \in X$.
    Since $q > 2$ is a power of $2$, by Remark~\ref{rem:nonsimple}, we have $\operatorname{wt}(X) \in (1-\zeta)\mathbb Z[\zeta]$ for each $X \in \mathcal O_2$.
    Therefore 
        \begin{align*}
        \sum_{X \in \mathcal O_2}\sum_{\mathbf w \in X}  \frac{2N \operatorname{wt}(\mathbf w)}{|X|} &= 
         \sum_{X \in \mathcal O_2}\sum_{\mathbf w \in U(X)}  N\frac{ \operatorname{wt}(\mathbf w)+\operatorname{wt}(\mathbf w^\top)}{|U(X)|} = \sum_{X \in \mathcal O_2} \operatorname{wt}(X) N
         \in (1-\zeta)N\mathbb Z[\zeta].
    \end{align*}
    By combining the above, we find that 
    \[
    \sum_{X \in \mathcal O}\sum_{\mathbf w \in X}  \frac{2N \operatorname{wt}(\mathbf w)}{|X|} = \sum_{X \in \mathcal O_0}\sum_{\mathbf w \in X}  \frac{2N \operatorname{wt}(\mathbf w)}{|X|} + \sum_{X \in \mathcal O_1}\sum_{\mathbf w \in X}  \frac{2N \operatorname{wt}(\mathbf w)}{|X|} + \sum_{X \in \mathcal O_2}\sum_{\mathbf w \in X}  \frac{2N \operatorname{wt}(\mathbf w)}{|X|} \in (1-\zeta)N\mathbb Z[\zeta],
    \]
    as required.
\end{proof}

In the next lemma, we write summands from Lemma~\ref{lem:weightedBurnside} in terms of the matrix $A=(J-H)/(1-\zeta)$.
Recall that $r$ and $s$ are the generators of $D_N = \langle r, s \; |\; r^N, s^2, (rs)^2 \rangle$ and define $\mathfrak W_N(H)$ as
\[
 \mathfrak W_N(H):= \left \{ \mathbf w \in \operatorname{fix}_{\Gamma(H)}(rs) \; \mid \; \frac{(1-\zeta)^2}{4}\operatorname{wt}(\mathbf w) \not \in (1-\zeta)\mathbb Z[\zeta] \right \}.
\]
Note that each walk $\mathbf w = (w_0,w_1,\dots,w_{N}) \in \mathfrak W_N(H)$ is not simple.
Indeed, $\mathbf w$ is fixed by $rs$, which means $w_0 = w_1$.

\begin{lemma}
    \label{lem:fixToTrace}
    Let $n, \in \mathbb N$, $q>2$, $H \in \mathcal H_n(q)$, $A=(J-H)/(1-\zeta)$, $\Gamma=\Gamma(H)$, and $N \geqslant 3$ be an integer.
    Then
    \begin{itemize}
\item[(i)] 
     $\displaystyle\sum_{\mathbf w \in \operatorname{fix}_{\Gamma}(r^k)}\operatorname{wt}(\mathbf w)= \operatorname{tr} \left( \left(A^{\circ \frac{N}{\gcd{(k,N)}}}\right)^{\gcd{(k,N)}} \right)$, for all $k \in \mathbb Z$;

\item[(ii)] 
     if $N$ is even, $\displaystyle\sum_{\mathbf w \in \operatorname{fix}_{\Gamma}(r^{2k}s)}\operatorname{wt}(\mathbf w)=\mathbf 1^\top ( A \circ A^\top )^\frac{N}{2} \mathbf 1 $, for all $k\in \mathbb Z$.

\item[(iii)]
     if $q$ is a power of $2$,
     \[
     \begin{cases} 
         \displaystyle \sum_{k=0}^{N-1}\sum_{\mathbf w \in \operatorname{fix}_{\Gamma}(r^{k}s)}\operatorname{wt}(\mathbf w) \in (1-\zeta)N\mathbb Z[\zeta],\; & \text{if $N$ is odd;}\\      
\\
    \displaystyle \sum_{k=0}^{N/2-1}\sum_{\mathbf w \in \operatorname{fix}_{\Gamma}(r^{2k+1}s)}\operatorname{wt}(\mathbf w) - N|\mathfrak W_N(H)| \in (1-\zeta)N\mathbb Z[\zeta],\; & \text{if $N$ is even.} 
    \end{cases}
     \]
\item[(iv)]
     if $q$ is not a power of $2$,
     \[
     \begin{cases} 
         \displaystyle \sum_{k=0}^{N-1}\sum_{\mathbf w \in \operatorname{fix}_{\Gamma}(r^{k}s)}\operatorname{wt}(\mathbf w) =0,\; & \text{if $N$ is odd;}\\      
\\
    \displaystyle \sum_{k=0}^{N/2-1}\sum_{\mathbf w \in \operatorname{fix}_{\Gamma}(r^{2k+1}s)}\operatorname{wt}(\mathbf w) =0,\; & \text{if $N$ is even.} 
    \end{cases}
     \]
     \end{itemize}

\end{lemma}
\begin{proof}
    Let $\mathbf w=(w_0,w_1,\dots,w_{N-1},w_N) \in \operatorname{fix}_{\Gamma}(g)$ for some $g\in D_N$.
    We split in to cases for the group element $g$.

    First suppose $g=r^k$ for some $k \in \mathbb Z$.
    Then $\operatorname{ord}(g)=\frac{N}{\gcd(k,N)}$.
    Therefore $g$ has $\gcd(k,N)$ cycles each of length $\operatorname{ord}(g)$.
    It follows that  
    $\mathbf w$ is a closed $\gcd(k,N)$-walk $\mathbf x$ repeated $\operatorname{ord}(g)$ times.
    Thus $\operatorname{wt}(\mathbf w)=\operatorname{wt}
    (\mathbf x)^{\operatorname{ord}(g)}$, from which, using Remark~\ref{rem:weightPowers}, part (i) follows. 
    
    Suppose that $N$ is even and $g=r^{2k}s$ for some $k\in \mathbb Z$. 
    Then $g$ consists of $N/2-1$ cycles of length 2 and two cycles of length 1. 
    Without loss of generality, we may assume $w_0$ and 
    $w_{\frac{N}{2}}$ are fixed by $g$.
    Then $\mathbf w$ is a closed $N$-walk made up of an $N/2$-walk $\mathbf x$ followed by its reverse walk $\mathbf x^\top$.
    Hence the sum of such walks is equal to the sum of each entry of the matrix $( A \circ A^\top )^{N/2}$.
    Whence we have part (ii).
    
    At this point, note that all walks in $\operatorname{fix}_{\Gamma}(g)$ must contain a loop if $N$ is odd and $g =r^ks$ or if $N$ is even and $g=r^{2k+1}s$ for some $k$.
    By Remark~\ref{rem:nonsimple}, if $q$ is not a power of $2$ then $\Gamma$ is a simple graph thus part (iv) follows immediately.
    It remains to assume that $q$ is a power of $2$.
    
    Suppose that $N$ is odd and $g=r^{k}s$ for some $k\in \mathbb Z$.
    Then $g$ consists of $(N-1)/2$ cycles of length $2$.
    This means that one pair of adjacent vertices of the walk $\mathbf w$ must be the same.    
    When $g=s$, we have $w_{\frac{N-1}{2}} = w_{\frac{N+1}{2}}$, that is, there is a loop at $w_{\frac{N-1}{2}}$.
    Thus $A_{w_{\frac{N-1}{2}},w_{\frac{N-1}{2}} } = 2/(1-\zeta)$.
    Note that, since $q>2$ is a power of $2$, by Proposition~\ref{pro:idealprelim} (f), we have $2/(1-\zeta) \in (1-\zeta)\mathbb Z[\zeta]$.
    Therefore, we can write $\operatorname{wt}(\mathbf w) = (1-\zeta)\alpha_{\mathbf w}$ for some $\alpha_{\mathbf w} \in \mathbb Z[\zeta]$.
    Observe that $\mathbf w \in \operatorname{fix}_{\Gamma}(r^{2i}s)$ for some $i \in \{0,\dots,N-1\}$ if and only if $\mathbf w^i \in \operatorname{fix}_{\Gamma}(s)$.
    Indeed, in both cases (with indices reduced modulo $N$) we have $w_{i+j} = w_{N+i-j}$ for all $j \in \{0,\dots,N-1\}$.
    Since $\operatorname{wt}(\mathbf w^i) = \operatorname{wt}(\mathbf w)$, it follows that
    $\sum_{\mathbf w \in \operatorname{fix}_{\Gamma}(r^{2i}s)}\operatorname{wt}(\mathbf w) = \sum_{\mathbf w \in \operatorname{fix}_{\Gamma}(s)}\operatorname{wt}(\mathbf w)$.
    Hence,

    \[
    \sum_{k=0}^{N-1}\sum_{\mathbf w \in \operatorname{fix}_{\Gamma}(r^{k}s)}\operatorname{wt}(\mathbf w) =
         \sum_{i=0}^{N-1}\sum_{\mathbf w \in \operatorname{fix}_{\Gamma}(r^{2i}s)}\operatorname{wt}(\mathbf w) 
         = N\sum_{\mathbf w \in \operatorname{fix}_{\Gamma}(s)}\operatorname{wt}(\mathbf w) 
        = N(1-\zeta)\sum_{\mathbf w \in \operatorname{fix}_{\Gamma}(s)}\alpha_{\mathbf w} \in (1-\zeta)N\mathbb Z[\zeta].
    \]
    
    Finally, suppose that $N$ is even and $g=r^{2k+1}s$ for some $k\in \mathbb Z$.
    Then $g$ consists of $N/2$ cycles of length $2$.
    This means that two pairs of adjacent vertices of the walk $\mathbf w$ must be the same.    
    When $g=rs$, we have $w_0 = w_1$ and 
    $w_{\frac{N}{2}} = w_{\frac{N}{2}+1}$.
    This means that there are loops at $w_0$ and $w_{\frac{N}{2}}$.
    Thus $A_{w_0,w_0} = A_{w_{\frac{N}{2}},w_{\frac{N}{2}}} = 2/(1-\zeta)$.
    Note that, since $q>2$ is a power of $2$, we have $4/(1-\zeta)^2 \in 2\mathbb Z[\zeta]$.
    Therefore, we can write $\operatorname{wt}(\mathbf w) = 2\alpha_{\mathbf w}$ for some $\alpha_{\mathbf w} \in \mathbb Z[\zeta]$.
    Observe that $\mathbf w \in \operatorname{fix}_{\Gamma}(r^{2k+1}s)$ if and only if $\mathbf w^k \in \operatorname{fix}_{\Gamma}(rs)$.
    Indeed, in both cases (with indices reduced modulo $N$) we have $w_{k+j} = w_{N+k-j}$ for all $j \in \{0,\dots,N/2-1\}$.
    Since $\operatorname{wt}(\mathbf w^k) = \operatorname{wt}(\mathbf w)$, it follows that
    $\sum_{\mathbf w \in \operatorname{fix}_{\Gamma}(r^{2k+1}s)}\operatorname{wt}(\mathbf w) = \sum_{\mathbf w \in \operatorname{fix}_{\Gamma}(rs)}\operatorname{wt}(\mathbf w)$.
    Hence,
    \[
    \sum_{k=0}^{N/2-1}\sum_{\mathbf w \in \operatorname{fix}_{\Gamma}(r^{2k+1}s)}\operatorname{wt}(\mathbf w) = \frac{N}{2}\sum_{\mathbf w \in \operatorname{fix}_{\Gamma}(rs)}\operatorname{wt}(\mathbf w) = N\sum_{\mathbf w \in \operatorname{fix}_{\Gamma}(rs)}\alpha_{\mathbf w}.
    \]
    By Proposition~\ref{pro:idealprelim} (g), we have $\sum_{\mathbf w \in \operatorname{fix}_{\Gamma}(rs)}\alpha_{\mathbf w} - |\mathfrak W_N(H)| \in (1-\zeta)\mathbb Z[\zeta]$.
    Whence we have part (iii).
\end{proof}

By combining Lemma~\ref{lem:weightedBurnside} and Lemma~\ref{lem:fixToTrace}, we arrive at the following corollaries.
These results should be compared with Theorem~\ref{thm:hs}.

\begin{corollary}
\label{cor:mainHSgenOdd}
Let $n \in \mathbb N$,
$q > 2$ be a power of $2$, $N \geqslant 3$ be an odd integer, $H \in \mathcal H_n(q)$, and $A=(J-H)/(1-\zeta)$.
 Then
  $$\sum_{d \;|\; N} \phi(N/d) \operatorname{tr}((A^{\circ N/d})^d) \in (1-\zeta)N \mathbb Z[\zeta].$$
\end{corollary}

For the case when $N$ is even, we first state a slightly weaker generalisation of 
Theorem~\ref{thm:hs},
which applies only to $q$-Seidel matrices.

\begin{corollary}
\label{cor:mainSeidelgen}
Let $n \in \mathbb N$,
$q$ be a power of $2$, $N \geqslant 4$ be an even integer, $S$ be a $q$-Seidel matrix, and $A=(J-I-S)/(1-\zeta)$.
 Then
  $$\sum_{d \;|\; N} \phi(N/d) \operatorname{tr}((A^{\circ N/d})^d) + \frac{N}{2}\mathbf 1^\top (A \circ A^\top)^{N/2} \mathbf 1 \in (1-\zeta)N \mathbb Z[\zeta].$$
\end{corollary}

As we will see below, the corresponding result for matrices $H \in \mathcal H_n(q)$ requires the use of the set $\mathfrak W_N(H)$.

Now, we state a slightly more general version of the above result, which we consider to be the main result of this section.
We will refine this result later in Section~\ref{sec:euler}, where we will deal with the unsightly term involving $|\mathfrak W_N(H)|$.

\begin{corollary}
\label{cor:mainHSgen}
Let $n \in \mathbb N$,
$q > 2$ be a power of $2$, $N \geqslant 4$ be an even integer, $H \in \mathcal H_n(q)$, and $A=(J-H)/(1-\zeta)$.
 Then
  $$\sum_{d \;|\; N} \phi(N/d) \operatorname{tr}((A^{\circ N/d})^d) + N|\mathfrak W_N(H)| +  \frac{N}{2}\mathbf 1^\top (A \circ A^\top)^{N/2} \mathbf 1 \in (1-\zeta)N \mathbb Z[\zeta].$$
\end{corollary}

We conclude this section with a preparatory result about the trace of powers of the matrix $A$.


\begin{lemma}
    \label{lem:trCong}
    Let $n \in \mathbb N$, $q > 2$ be a power of $2$, $H \in \mathcal H_n(q)$, and $A = (J-H)/(1-\zeta)$.
    Then, for all $k \geqslant 2$, there exists $U \subset W_k^\prime(\Gamma(H))$ such that
    \begin{align}
    \operatorname{tr}( A^k) - \sum_{\mathbf w \in U} \mathrm{wt}(\mathbf w)+\mathrm{wt}(\mathbf w^\top) &\in (1-\zeta)^2\mathbb Z[\zeta] \text{ and } \label{eqn:trAk} \\
    \operatorname{tr}( (A \circ A^\top)^k) - 2\sum_{\mathbf w \in U} \mathrm{wt}(\mathbf w)\mathrm{wt}(\mathbf w^\top) &\in 2(1-\zeta)\mathbb Z[\zeta]. \label{eqn:trAoAt}
    \end{align}
\end{lemma}
\begin{proof}
    Let $\Gamma = \Gamma(H)$ be the underlying graph of $H$.
    First, using Remark~\ref{rem:weightPowers}, observe that $\operatorname{tr}( A^k)$ is equal to the sum of $\operatorname{wt}(\mathbf w)$ over $\mathbf w \in W_k(\Gamma)$.
    Note that the diagonal entries of $A$ are equal to $0$ or $2/(1-\zeta) \in (1-\zeta)\mathbb Z[\zeta]$.
    Therefore, it is clear that the sum of the weights of walks that contain at least $2$ loops is in $(1-\zeta)^2\mathbb Z[\zeta]$.
    
    Let $W$ be the subset of $W_k(\Gamma)$ each of whose walks $\mathbf w$ contains a loop and $\operatorname{wt}(\mathbf w) \not \in (1-\zeta)^2\mathbb Z[\zeta]$.
    We claim that the sum of the weights of the walks in $W$ is in $(1-\zeta)^2\mathbb Z[\zeta]$.
    Each walk $\mathbf w \in W$ has the form
    \[
    \mathbf w = (w_0,w_1,\dots,w_{i-1},w_i,w_i,w_{i+1},\dots,w_{k-1}).
    \]
    We can obtain the simple walk $\mathbf x = (w_0,w_1,\dots,w_{i-1},w_i,w_{i+1},\dots,w_{k-1}) \in W_{k-1}^\prime(\Gamma)$ by removing the edge $(w_i,w_i)$.
    Note that $\operatorname{wt}(\mathbf w) = 2\operatorname{wt}(\mathbf x)/(1-\zeta)$.
    By Corollary~\ref{cor:UpartWalk}, there exists $U \subset W_{k-1}^\prime(\Gamma)$ such that $W_{k-1}^\prime(\Gamma) = U \sqcup \Psi(U)$.
    Let $L(\Gamma)$ be the set of vertices $v$ of $\Gamma$ such that $(v,v)$ is a loop and let $U(v)$ be the subset of $U$ consisting of the walks that contain the vertex $v$.
    Then
    \[
    \sum_{\mathbf w \in W} \operatorname{wt}(\mathbf w) - \frac{2}{1-\zeta}\sum_{v \in L(\Gamma)} \sum_{\mathbf x \in U(v)} \operatorname{wt}(\mathbf x) + \operatorname{wt}(\Psi(\mathbf x)) \in (1-\zeta)^2\mathbb Z[\zeta]; 
    \]
    Our claim then follows from Proposition~\ref{pro:idealprelim} (f) and Remark~\ref{rem:weightsum}.
    Therefore, 
    \[
    \operatorname{tr}( A^k) - \sum_{\mathbf w \in W_k^\prime(\Gamma)} \operatorname{wt}(\mathbf w) \in (1-\zeta)^2\mathbb Z[\zeta]
    \]
    and \eqref{eqn:trAk} follows from Corollary~\ref{cor:UpartWalk}.

    The proof of \eqref{eqn:trAoAt} is similar.
    We merely point out the differences to the above.
    Using Remark~\ref{rem:nonsimple}, we observe that $\operatorname{tr}( (A \circ A^\top)^k)$ equal to the sum of $\mathrm{wt}(\mathbf w)\mathrm{wt}(\mathbf w^\top)$ over $\mathbf w \in  W_k(\Gamma)$.
    Note that the diagonal entries of $A \circ A^\top$ are equal to $0$ or $4/(1-\zeta)^2 \in 2\mathbb Z[\zeta]$.
    If $\mathbf w \in  W_k(\Gamma)$ is not a simple walk then either $\operatorname{wt}(\mathbf w)\mathrm{wt}(\mathbf w^\top) \in 2(1-\zeta)\mathbb Z[\zeta]$ or $\operatorname{wt}(\mathbf w)\mathrm{wt}(\mathbf w^\top) = 4/(1-\zeta)^2 \operatorname{wt}(\mathbf x)\mathrm{wt}(\mathbf x^\top)$ for some $\mathbf x \in   W_{k-1}^\prime(\Gamma)$.
    We leave the rest of the details of the proof of \eqref{eqn:trAoAt} to the reader.
\end{proof}

\section{Determinants and characteristic polynomials}
\label{sec:dets}

In this section, we establish results about the determinants and characteristic polynomials of matrices in $\mathcal H_n(q)$.
At the end of this section we provide proofs of part (c) and part (d) of Theorem~\ref{thm:mainBounds} and Theorem~\ref{thm:mainBounds2}.

Denote by $\mathfrak S_n$ and $\mathfrak D_n$, respectively, the symmetric group and its subset of derangements acting on $\{1,\dots,n\}$.
We begin with a result about the determinant of a matrix closely related to matrices in $\mathcal H_n(q)$.

   \begin{lemma}
    \label{lem:detISMnodd}
    Let $n, q \in \mathbb N$ with $n$ odd, $q > 2$, $H \in \mathcal H_n(q)$, and $A = (J-H)/(1-\zeta)$.      
    Then $\det (A) \in (1-\zeta) \mathbb Z[\zeta]$.
\end{lemma}
\begin{proof}
    Using Leibniz's formula, we can write
    \begin{equation}
    \label{detA1}
        \det (A) = \sum_{\sigma \in \mathfrak S_n} \operatorname{sgn}(\sigma)\prod_{i=1}^n A_{i,\sigma(i)},
    \end{equation}
    where $\operatorname{sgn}(\sigma)$ denotes the signature of the permutation $\sigma$.
    Since $q>2$, each diagonal entry of $A$ is in $(1-\zeta)\mathbb Z[\zeta]$, if $\sigma \in \mathfrak S_n$ has a fixed point then $\prod_{i=1}^n A_{i,\sigma(i)} \in (1-\zeta)\mathbb Z[\zeta]$.
    Furthermore, since $n$ is odd, for each $\sigma \in \mathfrak D_n$, we have $\sigma \ne \sigma^{-1}$ and thus there exists a subset $\mathfrak U \subset \mathfrak D_n$ such that
    \[
        \mathfrak D_n = \mathfrak U \sqcup \{ \sigma^{-1} \; | \; \sigma \in \mathfrak U \}.
    \]
    Therefore, we can rewrite \eqref{detA1} as
    \begin{align*}
        \det( A )
        &= \sum_{\sigma \in \mathfrak U} \operatorname{sgn}(\sigma) \left( \prod_{i=1}^n A_{i,\sigma(i)}
        + \prod_{i=1}^n A_{\sigma(i),i} \right).
        %
    \end{align*}
    By Remark~\ref{rem:Aentries}, for each $\sigma \in \mathfrak U$, we have 
    \[
    \prod_{i=1}^n A_{\sigma(i),i} - (-1)^n\prod_{i=1}^n A_{i,\sigma(i)} \in (1-\zeta)\mathbb Z[\zeta].
    \]
    Since $n$ is odd, the $(-1)^n = -1$, whence we have the lemma.
\end{proof}

The statement of the next corollary follows immediately if one recalls the comments made after Lemma~\ref{lem:detIdealEasy}.

\begin{corollary}
     \label{cor:oddbi}
    Let $n, q \in \mathbb N$ with $q > 2$, $H \in \mathcal H_n(q)$, and $A = (J-H)/(1-\zeta)$.      
    Write $\operatorname{Char}_{A}(x)=\sum_{i=0}^n b_i x^{n-i}$.
    Then $b_i \in (1-\zeta) \mathbb Z[\zeta]$ for all  
    odd $i \in \{1,\dots,n\}$.
\end{corollary}

Next we show that the determinant of a matrix in $\mathcal H_n(q)$ lies in an ideal generated by a large power of $(1-\zeta)$.

\begin{lemma}
\label{lem:detIdeal}
     Let $n, q \in \mathbb N$ with $q > 2$ and $H \in \mathcal H_n(q)$.      
    Then $\det (H) \in (1-\zeta)^{n-1}\mathbb Z[\zeta]$.
    Furthermore, if $n$ is even then $\det( H) \in (1-\zeta)^{n}\mathbb Z[\zeta]$.
\end{lemma}

\begin{proof}
    Subtract the first row of $H$ from all other rows to obtain $H^\prime$.
    Except for the first row, each entry of $H^\prime$ is in $(1-\zeta)\mathbb Z[\zeta]$.
    Thus $\det (H) = \det (H^\prime) \in  (1-\zeta)^{n-1}\mathbb Z[\zeta]$.
    
    Now suppose $n$ is even.
    By taking the negative of $H$, if necessary, we may assume that $H_{1,1} = 1$.
    Conjugating by a diagonal matrix with entries in $\mathcal C_q$, if necessary, we may further assume that all entries in the first row and column of $H$ equal $1$.
    Then, except for $H^\prime_{1,1} = 1$, each entry of the first column of $H^\prime$ must equal $0$.
    Hence $\det(H^\prime) = \det(H^\prime[1])$.
    Observe that $H^\prime[1] = J-H[1] = (1-\zeta)A$ for some matrix $A$ of order $n-1$ whose entries belong to $\mathbb Z[\zeta]$.
    By Lemma~\ref{lem:detISMnodd}, the determinant of $A$ is in $(1-\zeta)\mathbb Z[\zeta]$.
    Thus, the lemma follows.
\end{proof}

Since they are Hermitian, each $H \in \mathcal H_n(q)$ must have $\det (H) \in \mathbb R$.  
The next corollary follows immediately from Lemma~\ref{lem:detIdeal} and Proposition~\ref{pro:realIdeal}.

\begin{corollary}
    \label{cor:detH}
    Let $n, q \in \mathbb N$ with $q > 2$, and let $H \in \mathcal H_n(q)$.
    Then $\det (H) \in \rho^{\lfloor n/2 \rfloor}\mathbb Z[\zeta+\zeta^{-1}]$.
\end{corollary}

The proof of the next lemma is a straightforward modification of the proof of \cite[Lemma 3.1]{GreavesYatsyna19}.

\begin{lemma}
     \label{lem:matdet}
    Let $A$ be a matrix of order $n$.
Write $\operatorname{Char}_{J-(1-\zeta)A}(x)=\sum_{i=0}^n a_ix^{n-i}$ and $\operatorname{Char}_A(x)=\sum_{i=0}^n b_ix^{n-i}$.
Then, for all $j \in \{0,\dots,n\}$,
$$a_j=(\zeta-1)^j \left( b_j + \frac{1}{1-\zeta} \sum_{i=1}^j b_{j-i} \mathbf{1}^\top A^{i-1}\mathbf{1} \right).$$
\end{lemma}

In Lemma~\ref{lem:matdet}, we see summands of the form $\mathbf 1^\top A^k \mathbf 1$. 
In preparation for Lemma~\ref{lem:armod}, we can obtain restrictions on these summands, as shown in the next result.

\begin{lemma}
    \label{lem:2sumApower}
    Let $n \in \mathbb N$, $q > 2$ be a power of $2$, $H \in \mathcal H_n(q)$, and $A = (J-H)/(1-\zeta)$.
    Then $\mathbf 1^\top A^k \mathbf 1 \in (1-\zeta)\mathbb Z[\zeta]$ for all $k \in \mathbb N$.
\end{lemma}
\begin{proof}
    First, since off-diagonal entries are counted twice and $2 \in (1-\zeta)\mathbb Z[\zeta]$ (by Proposition~\ref{pro:idealprelim} (f)), observe that 
    \[
    \mathbf 1^\top A^k \mathbf 1 - \operatorname{tr}(A^k) \in (1-\zeta)\mathbb Z[\zeta].
    \]
    Then the lemma follows from Lemma~\ref{lem:trCong}.
\end{proof}

Lemma~\ref{lem:matdet} and Lemma~\ref{lem:2sumApower} are the main tools used to obtain the following lemma about the coefficients of the characteristic polynomial of $H \in \mathcal H_n(q)$.

\begin{lemma}
    \label{lem:armod}
    Let $n \in \mathbb N$, $q > 2$ be a power of $2$, $H \in \mathcal H_n(q)$, and write $\operatorname{Char}_{H}(x)=\sum_{i=0}^n a_ix^{n-i}$.
    \begin{itemize}
        \item If $n$ is even then $a_j \in (1-\zeta)^j \mathbb Z[\zeta]$ for each $j \in \{0,\dots,n\}$.
        \item If $j  \in \{0,\dots,n\}$ is even then  $a_j \in (1-\zeta)^j \mathbb Z[\zeta]$.
    \end{itemize}
\end{lemma}
\begin{proof}
    The proof is trivial for $j = 0$.
    Assume $j \in \{1,\dots,n\}$.
    First, suppose $n$ is even.
    By Lemma~\ref{lem:matdet}, we can write
    $$a_j=(\zeta-1)^{j}b_{j} -(\zeta-1)^{j-1}nb_{j-1} - (\zeta -1)^{j-1}\sum_{i=2}^{j} b_{j-i} \mathbf{1}^\top A^{i-1}\mathbf{1}.$$
    Note that $n \in (1-\zeta)\mathbb Z[\zeta]$.
    By Lemma~\ref{lem:2sumApower}, we have 
    $\mathbf{1}^\top A^{i-1}\mathbf{1} \in (1-\zeta) \mathbb Z[\zeta]$ for all $i \geqslant 2$.
    Thus, $a_j \in (1-\zeta)^j \mathbb Z[\zeta]$ as required.
    
    Now suppose $j$ is even and write $j=2k$ for some positive integer $k$.
    By Lemma~\ref{lem:matdet}, we have
    \begin{align}
    \label{eqn:a2k}
    a_{2k} &=(\zeta-1)^{2k}b_{2k} +(1-\zeta)^{2k-1}b_{2k-1}n + (1-\zeta)^{2k-1}\sum_{i=2}^{2k} b_{2k-i} \mathbf{1}^\top A^{i-1}\mathbf{1}.
    \end{align}
    As $2k-1$ is odd, we use Corollary~\ref{cor:oddbi} to give $b_{2k-1} \in (1-\zeta)\mathbb Z[\zeta]$.
    Using Lemma~\ref{lem:2sumApower}, for each $i \geqslant 2$, we have  $\mathbf{1}^\top A^{i-1}\mathbf{1} \in (1-\zeta) \mathbb Z[\zeta]$.
    Thus, from \eqref{eqn:a2k}, we obtain $a_{2k} \in (1-\zeta)^{2k} \mathbb Z[\zeta]$ as required.
\end{proof}

Again, using the fact that each $H \in \mathcal H_n(q)$ is Hermitian, we have the following corollary, which is a consequence of Corollary~\ref{cor:detH}, and Lemma~\ref{lem:armod} together with Proposition~\ref{pro:realIdeal}. 

\begin{corollary}
    \label{cor:2powerCoeffs}
    Let $n \in \mathbb N$, $q > 2$, $H \in \mathcal H_n(q)$, and write $\operatorname{Char}_{H}(x)=\sum_{i=0}^n a_ix^{n-i}$.
        \begin{itemize}
    \item Suppose $q$ is not a power of $2$.
    Then $a_i \in \rho^{\lceil (i-1)/2 \rceil}\mathbb Z[\zeta+\zeta^{-1}]$ for all $i$.
    
    \item Suppose $q$ is a power of $2$.
    If $n$ is even then $a_i \in \rho^{\lceil i/2 \rceil}\mathbb Z[\zeta+\zeta^{-1}]$ for all $i$, otherwise $a_i \in \rho^{\lfloor i/2 \rfloor}\mathbb Z[\zeta+\zeta^{-1}]$ for all $i$.
    \end{itemize}
\end{corollary}

Now we are ready to prove parts (c) and (d) of Theorem~\ref{thm:mainBounds}.

\begin{proof}[Proof of Theorem~\ref{thm:mainBounds} (c)]
    Let $H \in \mathcal H_n(p^f)$ with $p$ an odd prime, $f \geqslant 1$, and write $\operatorname{Char}_{H}(x)=\sum_{i=0}^n a_ix^{n-i}$.
    By Remark~\ref{rem:a0a1a2}, $a_0 = 1$, $a_1 = -n$, and $a_2 = 0$.
    
    By Corollary~\ref{cor:2powerCoeffs} and Remark~\ref{rem:countResidues}, there are at most $p^{e-\lceil (i-1)/2 \rceil}$ possible residues for $a_i$ modulo $\rho^e \mathbb Z[\zeta+\zeta^{-1}]$, where $i \in \{3,\dots, 2e-1\}$ and just one possible residue for $a_i$ modulo $\rho^e \mathbb Z[\zeta+\zeta^{-1}]$ for $i \geqslant 2e$.
    
    Thus there are at most
    \[
        \prod_{i=3}^{2e-1}p^{e-\lceil (i-1)/2 \rceil} = p^{(e-1)^2}
    \]
    possible residue classes for $\operatorname{Char}_H(x)$ modulo $(\rho^e \mathbb Z[\zeta+\zeta^{-1}])[x]$.
\end{proof}

The proof of Theorem~\ref{thm:mainBounds} (d) is similar to the proof of Theorem~\ref{thm:mainBounds} (c).
The difference is that for $q$ an odd prime power, the trace of $H \in \mathcal H_n(q)$ is equal to $n$, which gives $a_1=-n$.

\begin{proof}[Proof of Theorem~\ref{thm:mainBounds} (d)]
    Let $H \in \mathcal H_n(2^f)$ where $n$ is even and $f > 1$ and write $\operatorname{Char}_{H}(x)=\sum_{i=0}^n a_ix^{n-i}$.
    By Remark~\ref{rem:a0a1a2}, $a_0 = 1$ and $a_2$ depends on $a_1$.
    Since $n$ is even, $a_1$ must also be an even integer.
    Using Remark~\ref{rem:rationalElements}, modulo $\rho^e \mathbb Z[\zeta+\zeta^{-1}]$, there are at most $2^{\lceil 2^{2-f}e \rceil - 1}$ residues for $a_1$ modulo $\rho^e \mathbb Z[\zeta+\zeta^{-1}]$.
    
    By Corollary~\ref{cor:2powerCoeffs} and Remark~\ref{rem:countResidues}, there are at most $2^{e-\lceil i/2 \rceil}$ possible residues for $a_i$ modulo $\rho^e \mathbb Z[\zeta+\zeta^{-1}]$, where $i \in \{3,\dots, 2e-1\}$ and just one possible residue for $a_i$ modulo $\rho^e \mathbb Z[\zeta+\zeta^{-1}]$ for $i \geqslant 2e$.
    
    Thus there are at most
    \[
        2^{\lceil 2^{2-f}e \rceil - 1}\prod_{i=3}^{2e-1}2^{e-\lceil i/2 \rceil} = 2^{(e-1)(e-2)+\lceil 2^{2-f}e \rceil - 1}
    \]
    possible residue classes for $\operatorname{Char}_H(x)$ modulo $(\rho^e \mathbb Z[\zeta+\zeta^{-1}])[x]$.
\end{proof}

\begin{proof}[Proof of Theorem~\ref{thm:mainBounds2} (c) and (d)]
    Note that $|\mathcal X_n^\prime(q,e)|$ is equal to the number of congruence classes of $\operatorname{Char}_H(x)$ modulo $\rho^e\mathbb Z[\zeta+\zeta^{-1}][x]$ where $H \in \mathcal H_n(q)$ and the diagonal entries of $H$ are all equal to $1$.
    This means that the top three coefficients of $\operatorname{Char}_{H}(x)=\sum_{i=0}^n a_ix^{n-i}$ are $a_0 = 1$, $a_1 = -n$, and $a_2=0$.
    From this point, the proof is similar to the proof of the corresponding part of  Theorem~\ref{thm:mainBounds}.
\end{proof}

\section{Residue graphs, Euler graphs, and Seidel switching}
\label{sec:euler}

In this section, we develop the final tools required to prove part (e) of Theorem~\ref{thm:mainBounds} and Theorem~\ref{thm:mainBounds2} and this section culminates with those proofs.
Accordingly, fix $q$ to be a power of $2$.

Let $H \in \mathcal H_n(q)$ and let $A = (J-H)/(1-\zeta)$.
A graph $\Gamma$ is called an \textbf{Euler graph} if the degree of each of its vertices is even.
The \textbf{residue graph} $\Gamma^{\mathrm{res}}(H)$ of $H$ is defined to have vertex set $\{1,\dots,n\}$ where $i$ is adjacent to $j$ if and only if $A_{i,j} \not \in (1-\zeta)\mathbb Z[\zeta]$.
It is easy to see that, for all $i$ and $j$, we have $A_{i,j} \in (1-\zeta)\mathbb Z[\zeta]$ if and only if $A_{i,j} \in (1-\zeta)\mathbb Z[\zeta]$, thus the adjacency of $\Gamma^{\mathrm{res}}(H)$ is well-defined.
Note that the residue graph $\Gamma^{\mathrm{res}}(H)$ is an Euler graph if and only if the sum of each row of $A$ is in $(1-\zeta) \mathbb Z[\zeta]$.
Furthermore, note that if $q > 2$ then $\Gamma^{\mathrm{res}}(H)$ is a simple graph (this follows from Proposition~\ref{pro:idealprelim} (f)).

Define $\operatorname{Diag}_n(\mathcal C_q)$ to be the set of diagonal $\mathcal C_q$-matrices of order $n$. 
The \textbf{switching class} $\operatorname{sw}(H)$ of $H$ is defined to be the following set of residue graphs
\[
   \operatorname{sw}(H) := \left \{ \Gamma^{\mathrm{res}} \left ( D HD^{-1}\right ) \; | \; D \in \operatorname{Diag}_n(\mathcal C_q) \right \}.
\]
Let $\Gamma = \Gamma^{\mathrm{res}} (H)$ with vertex set $V$, which is partitioned into two subsets $U_1$ and
$U_2$. 
\textbf{Seidel switching} with respect to $\{U_1, U_2\}$ is the operation on $\Gamma$ that leaves $V$ and the subgraphs induced by $U_1$ and $U_2$ unchanged, but deletes all edges between $U_1$ and $U_2$,
and inserts all edges between $U_1$ and $U_2$ that were not present in $\Gamma$.
Observe that by Seidel switching $\Gamma$ with respect to $\{U_1, U_2\}$, we obtain the residue graph of $DHD^{-1}$ where each entry of $D \in \operatorname{Diag}_n(\mathcal C_q)$ corresponding to $U_1$ is equal to $\zeta$ and equal to $1$ otherwise. 
Furthermore, as a consequence of Proposition~\ref{pro:idealprelim} (g), each graph in $\operatorname{sw}(H)$ can be obtained by performing some Seidel switching on the graph $\Gamma^{\mathrm{res}} (H)$.

We begin with a result about the presence of an Euler graph in the switching class of $H$.
The following proposition is generalisation of Seidel's corresponding result for $2$-Seidel matrices~\cite{Seidelnodd}.
We omit the proof since it is essentially the same as the proof of \cite[Theorem 1]{hage03}. 

\begin{proposition}
\label{pro:uniqueEuler}
    Let $n \in \mathbb N$ be odd, $q > 2$ be a power of $2$, and $H \in \mathcal H_n(q)$.
    Then $\operatorname{sw}(H)$ contains precisely one Euler graph.
\end{proposition}

In Lemma~\ref{lem:matdet}, we see summands of the form $\mathbf 1^\top A^k \mathbf 1$. 
When the residue graph of $H$ is an Euler graph we can obtain restrictions on these summands even stronger than those of Lemma~\ref{lem:2sumApower}, as shown in the next result.

\begin{lemma}
     \label{lem:sumEulerpower}
Let $n \in \mathbb N$, $q>2$ be a power of $2$, $H \in \mathcal H_n(q)$, and $A = (J-H)/(1-\zeta)$.
   Suppose that $\Gamma^{\mathrm{res}}(H)$ is an Euler graph.
   Then $\mathbf 1^\top A \mathbf 1 \in (1-\zeta) 
   \mathbb Z[\zeta]$ and $\mathbf 1^\top A^k \mathbf 1 
   \in (1-\zeta)^2 \mathbb Z[\zeta]$ for all $k \geqslant 
   2$.
\end{lemma}
\begin{proof}
    Since $\Gamma^{\mathrm{res}}(H)$ is an Euler graph, we can write $A\mathbf 
    1=(1-\zeta)\mathbf v$ for some $\mathbf v \in \mathbb Z[\zeta]^n$.
    Thus, for $k \geqslant 2$, we have $ \mathbf 1^\top A^k
    \mathbf 1=
    (1-\zeta^{-1})(1-\zeta)
    \mathbf v^\top A^{k-2} \mathbf v$, 
    which is in $(1-\zeta)^2 \mathbb Z[\zeta]$.
\end{proof}

In Corollary~\ref{cor:mainHSgen}, we see a term of the form $\mathbf 1^\top (A \circ A^\top)^k \mathbf 1$. 
Next, we show that this term is in $2(1-\zeta)\mathbb Z[\zeta]$ when the residue graph of $H$ is an Euler graph.

\begin{lemma}
    \label{lem:need2prove}
    Let $n \in \mathbb N$, $q > 2$ be a power of $2$, $H \in \mathcal H_n(q)$, and $A = (J-H)/(1-\zeta)$.
    Suppose that $\Gamma^{\mathrm{res}}(H)$ is an Euler graph.
    Then $\mathbf 1^\top (A \circ A^\top)^k \mathbf 1 \in 2(1-\zeta) \mathbb Z[\zeta]$ for all $k \geqslant 2$.
\end{lemma}

\begin{proof}
    Note that $\mathbf 1^\top (A \circ A^\top)^k \mathbf 1$ is the sum of $\mathrm{wt}(\mathbf w)\mathrm{wt}(\mathbf w^\top)$ of each walk $\mathbf w$ of $\Gamma(H)$ of length $k$.
    Let $M$ be the $\{0,1\}$-adjacency matrix of $\Gamma^{\mathrm{res}}(H)$.
    For each $i,j \in \{1,\dots,n\}$, 
    by Proposition~\ref{pro:idealprelim} (g), we have $A_{i,j} - M_{i,j} \in (1-\zeta)\mathbb Z[\zeta]$.
    Hence, we have 
    \begin{equation}
    \label{eqn:AtoM}
        ((A \circ A^\top)^k)_{i,j} - (M^k)_{i,j} \in (1-\zeta)\mathbb Z[\zeta].
    \end{equation}
    By Lemma~\ref{lem:trCong}, there exists $U \subset W_k^\prime(\Gamma(H))$ such that
    $$
    \operatorname{tr}( (A \circ A^\top)^k) - 2\sum_{\mathbf w \in U} \mathrm{wt}(\mathbf w)\mathrm{wt}(\mathbf w^\top) \in 2(1-\zeta)\mathbb Z[\zeta].
    $$
    Furthermore, using \eqref{eqn:AtoM}, we have
    \begin{equation}
    \label{eqn:trAtotrM}
        \frac{\operatorname{tr}( (A \circ A^\top)^k)}{2} - \frac{\operatorname{tr}(M^k)}{2}  \in (1-\zeta)\mathbb Z[\zeta].
    \end{equation}
    Since $ (A \circ A^\top)^k$ is symmetric, using \eqref{eqn:AtoM} together with \eqref{eqn:trAtotrM}, we find that
    $$\frac{\mathbf 1^\top (A \circ A^\top)^k \mathbf 1}{2} - \frac{\mathbf 1^\top M^k \mathbf 1}{2} \in (1-\zeta)\mathbb Z[\zeta].$$
    The result then follows from \cite[Lemma 2.6]{GreavesYatsyna19} since $\Gamma^{\mathrm{res}}(H)$ is an Euler graph.
\end{proof}

Recall from Section~\ref{sec:graphs} the definition of $\mathfrak W_N(H)$:
\[
 \mathfrak W_N(H):= \left \{ \mathbf w \in \operatorname{fix}_{\Gamma(H)}(rs) \; \mid \; \frac{(1-\zeta)^2}{4}\operatorname{wt}(\mathbf w) \not \in (1-\zeta)\mathbb Z[\zeta] \right \},
\]
where $r$ and $s$ are the generators of $D_N = \langle r, s \; |\; r^N, s^2, (rs)^2 \rangle$.
Before we establish our next main tool (Lemma~\ref{lem:trCongBurnside}, below) we need one more ingredient.

\begin{lemma}
\label{lem:EulerW}
    Let $n \in \mathbb N$, $q > 2$ be a power of $2$, $N \geqslant 4$ be an even integer, $H \in \mathcal H_n(q)$, and $A = (J-H)/(1-\zeta)$.
    Suppose that $\Gamma^{\mathrm{res}}(H)$ is an Euler graph.
    Then $|\mathfrak W_N(H)|$ is even.
\end{lemma}
\begin{proof}
    Let $\mathbf w = (w_0,w_1,\dots,w_N) \in \mathfrak W_N(H)$.
    Then $w_0 = w_1$, $w_{N/2}=w_{N/2+1}$, and furthermore $w_{N-i}=w_{i+1}$ for all $i \in  \{0,\dots,N/2\}$.
    Therefore, $\mathbf w = (w_0,w_0)\mathbf x \mathbf x^\top (w_{N/2},w_{N/2})$, where $\mathbf x$ is a walk of length $N/2-1$ from $w_0$ to $w_{N/2}$ and concatenation of walks is as defined in the proof of Lemma~\ref{lem:evenPalin}.
    Since $\mathrm{wt}(\mathbf w) = 4\mathrm{wt}(\mathbf x)\mathrm{wt}(\mathbf x^\top)/(1-\zeta)^2$ and $\mathbf w \in \mathfrak W_N(H)$, we must have $\mathrm{wt}(\mathbf x) \not \in (1-\zeta)\mathbb Z[\zeta]$.
    Hence each walk $\mathbf w = (w_0,w_1,\dots,w_{N/2},w_{N/2+1},\dots,w_N) \in \mathfrak W_N(H)$ corresponds to a walk $\mathbf x = (w_1,\dots,w_{N/2})$ of length $N/2-1$ in $\Gamma^{\mathrm{res}}(H)$, where both the vertices $w_0=w_1$ and $w_{N/2}=w_{N/2+1}$ have loops.
    
    Now, if $\mathbf x$ and $\mathbf x^\top$ are distinct, we can pair up $\mathbf x$ with $\mathbf x^\top$.
    It, therefore, remains to show that the number of palindromic walks of length $N/2-1$ in $\Gamma^{\mathrm{res}}(H)$ starting and ending at a vertex with a loop in $\Gamma(H)$ is even.
    In fact, we shall show that for all vertices $v$, the number of palindromic walks of length $N/2-1$ in $\Gamma^{\mathrm{res}}(H)$ starting and ending at $v$ is even.
    
    Let $P_{N/2-1}(v)$ be the set of palindromic walks in $\Gamma^{\mathrm{res}}(H)$ starting and ending at $v$.
    First, if $N/2-1$ is odd, then it is obvious that $P_{N/2-1}(v)$ is empty.
    Now suppose $k = N/2-1$ is even.
    Note that $|P_{2}(v)|$ is equal to the degree of the vertex $v$, which is even since $\Gamma^{\mathrm{res}}(H)$ is an Euler graph.
    It remains to assume that $k \geqslant 4$.
    Let $\mathbf p = (p_0,\dots,p_{k/2-1},p_{k/2},p_{k/2-1},\dots,p_0) \in P_{k}(v)$.
    Then the walk $(p_0,\dots,p_{k/2-1},\dots,p_0)$ obtained by removing the vertex $p_{k/2}$ and its incident edge to $p_{k/2-1}$ from the walk $\mathbf p$ is in $P_{k-2}(v)$.
    Thus, each walk in $P_{k}(v)$ can be obtained from a walk in $P_{k-2}(v)$.
    Furthermore, each walk in $P_{k-2}(v)$ can be extended to an even number of walks in $P_{k}(v)$.
    Indeed, since every vertex of $\Gamma^{\mathrm{res}}(H)$ has even degree, there are always an even number of choices for vertices $p_{k/2}$ adjacent to $p_{k/2-1}$ when extending the walk $(p_0,\dots,p_{k/2-1},\dots,p_0)$ to $(p_0,\dots,p_{k/2-1},p_{k/2},p_{k/2-1},\dots,p_0)$.
    This completes our proof.
\end{proof}

Now we can use Lemma~\ref{lem:need2prove} and Lemma~\ref{lem:EulerW} to obtain the following refinement of Corollary~\ref{cor:mainHSgen}. 

\begin{lemma}
\label{lem:trCongBurnside}
    Let $n \in \mathbb N$, $q > 2$ be a power of $2$, $N \geqslant 3$ be an integer, $H \in \mathcal H_n(q)$, and $A = (J-H)/(1-\zeta)$.
    Suppose that $\Gamma^{\mathrm{res}}(H)$ is an Euler graph.
    Then
  $$\sum_{d \;|\; N} \phi(N/d) \operatorname{tr}((A^{\circ N/d})^d) \in (1-\zeta)N \mathbb Z[\zeta].$$
\end{lemma}
\begin{proof}
    If $N$ is odd then the lemma is precisely Corollary~\ref{cor:mainHSgenOdd}.
    Otherwise, suppose that $N$ is even.
    By Corollary~\ref{cor:mainHSgen}, we have
    $$\sum_{d \;|\; N} \phi(N/d) \operatorname{tr}((A^{\circ N/d})^d) + N|\mathfrak W_N(H)| +  \frac{N}{2}\mathbf 1^\top (A \circ A^\top)^{N/2} \mathbf 1 \in (1-\zeta)N \mathbb Z[\zeta].$$
    The lemma then follows from Lemma~\ref{lem:need2prove} and Lemma~\ref{lem:EulerW} using Proposition~\ref{pro:idealprelim} (f).
\end{proof}

In preparation for an application of Lemma~\ref{lem:trCongBurnside}, we state and prove the following technical lemmas.
The first of these is a slight strengthening of Remark~\ref{rem:weightsum}.

\begin{lemma}
    \label{lem:wtwCong}
    Let $n, k \in \mathbb N$, $q > 2$ be a power of $2$, $k$ be odd, $H \in \mathcal H_n(q)$, $A = (J-H)/(1-\zeta)$ and $\mathbf w \in W_k^\prime(\Gamma)$.
    Then 
    \[
    \frac{\operatorname{wt}(\mathbf w) + \operatorname{wt}(\mathbf w^{\top})}{1-\zeta} - \operatorname{wt}(\mathbf w) \in (1-\zeta)\mathbb Z[\zeta].
    \]
\end{lemma}
\begin{proof}
    Write $\mathbf w = (w_0,w_1,\dots,w_k)$ and $\operatorname{wt}(\mathbf w) = \prod_{i=1}^{k} \frac{1-\zeta^{e_i}}{1-\zeta}$.
    The lemma is trivial if any of the $e_i$ is equal to $0$.
    Without loss of generality, we can assume that 
    $1 \leqslant e_i \leqslant q-1$ for each $i \in \{1,\dots,k\}$.
    Hence, $(1-\zeta^{e_i})/(1-\zeta) = 1+\zeta + \dots + \zeta^{e_i-1}$.
    Furthermore, 
    \[
    \operatorname{wt}(\mathbf w^\top) = \prod_{i=1}^{k} \frac{1-\zeta^{q-e_i}}{1-\zeta}= \prod_{i=1}^{k} (1+\zeta + \dots + \zeta^{q-e_i-1}).
    \]
    Next, for $e \in \mathbb N$, observe that
    \[
    \zeta^e - e\zeta + e-1 \in (1-\zeta)^2\mathbb Z[\zeta].
    \]
    It follows that
    \[
    (1-\zeta^{e})/(1-\zeta) - \left (\binom{e}{2}\zeta -\binom{e-1}{2} +1 \right)\in (1-\zeta)^2\mathbb Z[\zeta].
    \]
    By Proposition~\ref{pro:idealprelim} (f), we find that
    \[
    \begin{cases}
        (1-\zeta^{e})/(1-\zeta)  \in (1-\zeta)^2\mathbb Z[\zeta], & \text{ if $e \equiv 0 \pmod 4$}; \\
        (1-\zeta^{e})/(1-\zeta) -1  \in (1-\zeta)^2\mathbb Z[\zeta], & \text{ if $e \equiv 1 \pmod 4$}; \\
        (1-\zeta^{e})/(1-\zeta) - (1+\zeta)  \in (1-\zeta)^2\mathbb Z[\zeta], & \text{ if $e \equiv 2 \pmod 4$}; \\
        (1-\zeta^{e})/(1-\zeta) -\zeta  \in (1-\zeta)^2\mathbb Z[\zeta], & \text{ if $e \equiv 3 \pmod 4$}. \\
    \end{cases}
    \]

    If $e_i \equiv 0 \pmod 4$ for some $i \in \{1,\dots,k\}$ then both $\operatorname{wt}(\mathbf w) + \operatorname{wt}(\mathbf w^\top) \in (1-\zeta)^2\mathbb Z[\zeta]$ and $\operatorname{wt}(\mathbf w) \in (1-\zeta)^2\mathbb Z[\zeta]$.
    Therefore, it remains to consider the case when  $e_i \not \equiv 0 \pmod 4$ for all $i \in \{1,\dots,k\}$.
    In this case, we have
    \[
    \operatorname{wt}(\mathbf w) - 1^a(1+\zeta)^b\zeta^c \in (1-\zeta)^2\mathbb Z[\zeta],
    \]
    for some $a,b,c \in \mathbb N \cup \{0\}$.
    At the same, we have 
    \[
        \operatorname{wt}(\mathbf w^\top) - 1^c(1+\zeta)^b\zeta^a \in (1-\zeta)^2\mathbb Z[\zeta].
    \]
    Hence,
    \[
    \operatorname{wt}(\mathbf w) + \operatorname{wt}(\mathbf w^\top) - (1+\zeta)^b(\zeta^c+\zeta^a) \in (1-\zeta)^2\mathbb Z[\zeta].
    \]
    Using Proposition~\ref{pro:idealprelim} (g), we see that, if $b \geqslant 1$ then both $\operatorname{wt}(\mathbf w) + \operatorname{wt}(\mathbf w^\top) \in (1-\zeta)^2\mathbb Z[\zeta]$ and $\operatorname{wt}(\mathbf w) \in (1-\zeta)^2\mathbb Z[\zeta]$.

    Lastly, we can assume that $b = 0$.
    It suffices to show that $(\zeta^c+\zeta^a) - (1-\zeta)\zeta^c \in (1-\zeta)^2\mathbb Z[\zeta]$.
    This follows since $a+c = k$, which is odd.
\end{proof}

The second technical lemma is a relation for certain traces.

\begin{lemma}
\label{lem:keyLemTrCong}
    Let $n, k \in \mathbb N$, $q > 2$ be a power of $2$, $k \geqslant 3$ be odd, $H \in \mathcal H_n(q)$, and $A = (J-H)/(1-\zeta)$.
    Then 
    \[
    \operatorname{tr} ((A^{\circ 2})^k) - \operatorname{tr} (A^k)^2+2(1-\zeta)^{-1} \operatorname{tr} (A^k) \in 2(1-\zeta)\mathbb Z[\zeta].
    \]
\end{lemma}
\begin{proof}
    Let $\Gamma = \Gamma(H)$ be the underlying graph of $H$ and let $W =W_k(\Gamma)$ be the set of closed $k$-walks in $\Gamma$.
    First, observe that $\operatorname{tr} ((A^{\circ 2})^k) = \sum_{\mathbf w \in W} \operatorname{wt}(\mathbf w)^2$ and $\operatorname{tr} (A^{k}) = \sum_{\mathbf w \in W} \operatorname{wt}(\mathbf w)$.
    Therefore, we have
        \begin{equation}
        \label{eqn:NI}
    \operatorname{tr} ((A^{\circ 2})^k) = \operatorname{tr} (A^{k})^2-2\sum_{\{\mathbf w_1,\mathbf w_2\} \in \binom{W}{2}} \operatorname{wt}(\mathbf w_1)\operatorname{wt}(\mathbf w_2).
    \end{equation}
    Next, by Proposition~\ref{pro:Idealcor}, for each $\mathbf w \in W$, either $\operatorname{wt}(\mathbf w) \in (1-\zeta)\mathbb Z[\zeta]$ or $\operatorname{wt}(\mathbf w)-1 \in (1-\zeta)\mathbb Z[\zeta]$.
    Therefore 
    \[
        \sum_{\{\mathbf w_1,\mathbf w_2\} \in \binom{W}{2}} \operatorname{wt}(\mathbf w_1)\operatorname{wt}(\mathbf w_2) - \binom{t}{2} \in (1-\zeta)\mathbb Z[\zeta],
    \]
    where $t$ is the number of walks $\mathbf w \in W$ satisfying $\operatorname{wt}(\mathbf w)-1 \in (1-\zeta)\mathbb Z[\zeta]$.
    Thus, in view of \eqref{eqn:NI}, it remains to show that $(1-\zeta)^{-1} \operatorname{tr} (A^k)- \binom{t}{2} \in (1-\zeta)\mathbb Z[\zeta]$.
    Let $T$ be the set of $\mathbf w \in W$ satisfying $\operatorname{wt}(\mathbf w)-1 \in (1-\zeta)\mathbb Z[\zeta]$.
    Then $|T| = t$.
    Furthermore, by Remark~\ref{rem:nonsimple}, we must have $T \subset W_k^\prime(\Gamma)$.

    Since $k$ is odd, for each $\mathbf w \in T$, the reverse walk $\mathbf w^{\top}$ is distinct from $\mathbf w$ and is also in $T$.
    Thus $t$ must be even and, moreover, $\binom{t}{2} - \frac{t}{2} \in (1-\zeta)\mathbb Z[\zeta]$.
    By Lemma~\ref{lem:trCong}, there exists a subset $U \subset W_k^\prime(\Gamma) \subset W$ such that
    \begin{equation}
        \label{eqn:1-ztr1}
        \operatorname{tr} (A^k) - \sum_{\mathbf w \in U}\left ( \operatorname{wt}(\mathbf w) + \operatorname{wt}(\mathbf w^{\top}) \right ) \in (1-\zeta)^2\mathbb Z[\zeta].
    \end{equation}
    By Remark~\ref{rem:weightsum}, for each $\mathbf w \in U$, we have $\operatorname{wt}(\mathbf w) + \operatorname{wt}(\mathbf w^{\top}) \in (1-\zeta)\mathbb Z[\zeta]$.
    Furthermore, by Lemma~\ref{lem:wtwCong}, we can write
    \begin{equation}
        \label{eqn:1-ztr2}
        \frac{\operatorname{wt}(\mathbf w) + \operatorname{wt}(\mathbf w^{\top})}{1-\zeta} - \operatorname{wt}(\mathbf w) \in (1-\zeta)\mathbb Z[\zeta].
    \end{equation}
    Combining \eqref{eqn:1-ztr1} and \eqref{eqn:1-ztr2} yields
    \[
    (1-\zeta)^{-1} \operatorname{tr} (A^k)- \sum_{\mathbf w \in U}\operatorname{wt}(\mathbf w) \in (1-\zeta)\mathbb Z[\zeta].
    \]
    Thus, since $|U\cap T| = t/2$, we obtain
    $(1-\zeta)^{-1} \operatorname{tr} (A^k)- \frac{t}{2} \in (1-\zeta)\mathbb Z[\zeta]$, as required.
\end{proof}

Now we are ready to establish a lemma, which acts as a key tool in the proof of the last part of the main theorems.

\begin{lemma}
\label{lem:quadtracCong}
Let $n \in \mathbb N$, $q > 2$ be a power of $2$, $H \in \mathcal H_n(q)$, $A = (J-H)/(1-\zeta)$, and
    let $N>2$ be an integer congruent to $2$ modulo $4$.
    Suppose that $\Gamma^{\mathrm{res}}(H)$ is an Euler graph.
    Then
$$\operatorname{tr} (A^{N/2})^2-2(1-\zeta)^{-1} \operatorname{tr} (A^{N/2}) + \operatorname{tr}(A^N) \in 2(1-\zeta) \mathbb Z[\zeta].$$
\end{lemma}
\begin{proof}
    By Lemma~\ref{lem:keyLemTrCong}, using $k = N/2$, it suffices to show that
    \[
    \operatorname{tr}(A^N) + \operatorname{tr} ((A^{\circ 2})^{N/2}) \in 2(1-\zeta) \mathbb Z[\zeta].
    \]
    By Lemma~\ref{lem:trCongBurnside}, we have
    \begin{align}
    \label{eqn:sumtr}
        \sum_{d \;|\; N} \phi(N/d) \operatorname{tr}((A^{\circ N/d})^d) \in 2(1-\zeta) \mathbb Z[\zeta].
    \end{align}
    Except for $d = N$ and $d = N/2$, for each divisor $d$ of $N$, we have that $\phi(N/d)$ is even.
    Observe that 
    \[
    \operatorname{tr}((A^{\circ k})^d) - \operatorname{tr}(A^d) \in (1-\zeta) \mathbb Z[\zeta]
    \]
    for all $k$ and $d$ in $\mathbb N$.
    Indeed, $\operatorname{tr}(A^d)$ is the sum of the weights of the walks in $W_d(\Gamma(H))$ and  $\operatorname{tr}((A^{\circ k})^d)$ is the sum of their $k$th powers.
    For each $\mathbf w \in W_d(\Gamma(H))$, we have $\operatorname{wt}(\mathbf w)^k - \operatorname{wt}(\mathbf w) \in (1-\zeta) \mathbb Z[\zeta]$.
    Hence, using Lemma~\ref{lem:trCong} together with \eqref{eqn:sumtr}, it follows that 
    \[
    \operatorname{tr}(A^N) + \operatorname{tr} ((A^{\circ 2})^{N/2}) \in 2(1-\zeta) \mathbb Z[\zeta],
    \]
    as required.
\end{proof}

For $H \in \mathcal H_n(q)$ and $A = (J-H)/(1-\zeta)$, we have the following relation between coefficients of the characteristic polynomials of $H$ and $A$.

\begin{lemma}
     \label{lem:a_kAndb_k}
     Let $n \in \mathbb N$ be odd, $q > 2$ be a power of $2$, $H \in \mathcal H_n(q)$, and $A = (J-H)/(1-\zeta)$.
    Suppose that $\Gamma^{\mathrm{res}}(H)$ is an Euler graph.
    Write $\operatorname{Char}_{H}(x)=\sum_{i=0}^n a_i 
    x^{n-i}$ and $\operatorname{Char}_A(x)=\sum_{i=0}^n b_i x^{n-i}$.
    Then for all $j \in \{1,\dots ,\frac{n-1}{2} \}
    $, we have
$$ b_{2j} - \frac{a_{2j+1}}{(1-\zeta)^{2j}}
\in (1-\zeta)^2 \mathbb Z[\zeta]; $$
$$ b_{2j-1}- \frac{a_{2j+1}+a_{2j} + a_{2j-1}(a_3 + a_2 - a_1 - n)}{(1-\zeta)^{2j-1}} \in (1-\zeta)^2\mathbb Z[\zeta]. $$
\end{lemma}
\begin{proof}
   By Lemma~\ref{lem:matdet}, we can write
   \begin{equation*}
       \frac{a_{2j+1}}{(1-\zeta)^{2j}}=(\zeta-1)b_{2j+1} - nb_{2j} - \sum_{i=2}^{2j+1} b_{2j+1-i} \mathbf{1}^\top A^{i-1}\mathbf{1}.
   \end{equation*}
   By Corollary~\ref{cor:oddbi}, we know that $b_{2i+1} \in (1-\zeta)\mathbb Z[\zeta]$ for all $i \in \{1,\dots,(n-1)/2\}$.
   By Lemma~\ref{lem:sumEulerpower}, we have
   $\mathbf{1}^\top A\mathbf{1} \in (1-\zeta)\mathbb Z[\zeta]$ and  $\mathbf{1}^\top A^i\mathbf{1} \in (1-\zeta)^2\mathbb Z[\zeta]$ for all $i \geqslant 2$.
   Putting to above together with the fact that $2\mathbb Z[\zeta] \subset (1-\zeta)^2\mathbb Z[\zeta]$ (see Proposition~\ref{pro:idealprelim} (f)) yields
   \begin{equation}
       \label{eqn1}
   b_{2j} - \frac{a_{2j+1}}{(1-\zeta)^{2j}} \in 
   (1- \zeta)^2 \mathbb Z[\zeta].
   \end{equation}
   Next, using Lemma~\ref{lem:matdet}, we 
   have
   $a_2=(1-\zeta)^2 b_2 + (1-\zeta)b_1 n + (1-
   \zeta)b_0 \mathbf 1^\top A \mathbf 1$ and 
   $$a_3=(\zeta-1)^3 b_3 - (1-\zeta)^2 b_2 n - (1-
   \zeta)^2 b_1\mathbf 1^\top A \mathbf 1 - (1-
   \zeta)^2 b_0 \mathbf 1^\top A^2 \mathbf 1.$$
   Since $q>2$, each diagonal entry of $A$ belongs to $(1-\zeta)\mathbb Z[\zeta]$, we find that $\operatorname{tr} (A) = -b_1 \in (1-\zeta)\mathbb Z[\zeta]$.
   By Lemma~\ref{lem:armod}, we have $a_2 \in (1-\zeta)^2\mathbb Z[\zeta]$.
   Thus, using Lemma~\ref{lem:sumEulerpower} and Proposition~\ref{pro:idealprelim} (f), we can write
   \begin{equation}
       \frac{a_3+a_2}{1-\zeta} -(b_1 + \mathbf 1^\top A 
   \mathbf 1) \in (1-\zeta)^2 \mathbb Z[\zeta]. \label{eqn:fI}
   \end{equation}
   Using \eqref{eqn1}, we have
   \begin{equation}
        b_{2j-2} - \frac{a_{2j-1}}{(1-\zeta)^{2j-2}} \in 
(1-\zeta)^2 \mathbb Z[\zeta]. \label{eqn:fII}
   \end{equation}
   Take the trace of $A = (J-H)/(1-\zeta)$ to obtain the identity $b_1 = (n+a_1)/(\zeta-1)$.
   Combining this identity with \eqref{eqn:fI} and \eqref{eqn:fII}, we obtain
\begin{equation}
\label{eqn:fIII}
    \frac{(a_2+a_3)a_{2j-1}}{(1-\zeta)^{2j-1}} - b_{2j-2} \left (\mathbf 1^\top A 
   \mathbf 1 - \frac{a_1+n}{1-\zeta}\right ) \in (1-\zeta)^2 \mathbb Z[\zeta].
\end{equation}
Using Lemma~\ref{lem:matdet} and Lemma~\ref{lem:sumEulerpower}, we can write
   \begin{equation*}
       \frac{a_{2j+1}+a_{2j}}{(1-\zeta)^{2j-1}} - \left ( (1-n)b_{2j} - nb_{2j-1}-b_{2j-2}\mathbf{1}^\top A\mathbf{1} \right ) \in (1-\zeta)^2\mathbb Z[\zeta].
   \end{equation*}
   Apply Proposition~\ref{pro:idealprelim} (f) to obtain
      \begin{equation*}
       \frac{a_{2j+1}+a_{2j}}{(1-\zeta)^{2j-1}} - \left (  b_{2j-1}-b_{2j-2}\mathbf{1}^\top A\mathbf{1} \right ) \in (1-\zeta)^2\mathbb Z[\zeta].
   \end{equation*}
Combining with \eqref{eqn:fII} and \eqref{eqn:fIII}, yields
   \begin{equation*}
       b_{2j-1} - \frac{a_{2j+1}+a_{2j}+a_{2j-1}(a_3+a_2-a_1-n)}
       {(1-\zeta)^{2j-1}} \in (1-\zeta)^2\mathbb Z[\zeta]. \qedhere
   \end{equation*}
\end{proof}

For a positive integer $a$, the $2$\textbf{-adic valuation} $\nu_2(a)$ of $a$ is the multiplicity of $2$ in the prime factorisation of $a$.
The $2$-adic valuation of $0$ is defined to be $\infty$.
For an integer $b$ relatively prime to $a$, the $2$-adic valuation $\nu_2(a/b)$ of $a/b$ is defined as
$-\nu_2(b)$ if $b$ is even and $\nu_2(a)$ otherwise.
We now state a couple of lemmas about the $2$-adic valuation.

\begin{lemma}[{\cite[Lemma 3.11]{GreavesYatsyna19}}]
    \label{lem:2valbound}
    Let $l$ be a positive integer and let $m_1,m_2,\dots , 
    m_l$ be nonnegative integers having a positive sum.
    Let $m \in \{m_i\ |\ m_i \neq 0 \}$.
    Then
$$\nu_2 \left( \frac{(m_1+m_2+\dots +m_l-1)!}{m_1!m_2!\dots m_l!} \right) \geqslant -\nu_2(m).$$
\end{lemma}

\begin{lemma}
    \label{lem:2adicdiff}
    Let $l$ be a positive integer and let $m_1,m_2,\dots, m_l$ be nonnegative integers having a positive sum $m$.
    Then
$$ \nu_2 \left( \frac{m!}{m_1!m_2!\dots m_l!} \right)=\nu_2 \left( \frac{(2m)!}{(2m_1)!(2m_2)!\dots (2m_l)!} \right).$$
\end{lemma}
\begin{proof}
    Set
    \[
        D = \nu_2 \left( \frac{m!}{m_1!m_2!\dots 
    m_l!} \right)-\nu_2 \left( \frac{(2m)!}
    {(2m_1)!(2m_2)!\dots (2m_l)!} \right).
    \]
    Using Legendre's formula~\cite[Page 77]{Moll12}, we have
    \begin{align*}
    D &= \nu_2 (m!)- \nu_2 ((2m)!) + \sum_{i=1}^l \nu_2 
    ((2m_i)!)- \nu_2 ((m_i)!)\\
    &= \sum_{j=1}^\infty \left\lfloor \frac{m}{2^j} \right
    \rfloor - \left\lfloor \frac{2m}{2^j} 
    \right\rfloor + \sum_{i=1}^l \sum_{j=1}^\infty \left
    \lfloor \frac{2m_i}{2^j} \right\rfloor -
    \left\lfloor \frac{m_i}{2^j} \right\rfloor \\
    &= -m 
    + \sum_{i=1}^l \left (m_i + \sum_{j=1}^
    \infty \left \lfloor \frac{2m_i}{2^{j+1}} \right\rfloor -
    \left\lfloor \frac{m_i}{2^j} \right\rfloor \right ) = -m + m = 0. \qedhere
    \end{align*} 
\end{proof}

Define the set $X(d)$ as the set of nonnegative integral vectors $\mathbf x = (x_1,\dots,x_d) \in \mathbb Z_{\geqslant 0}^d$ satisfying the constraint
\[
    x_1 + 2x_2 + \dots + dx_d = d.
\]
For each $\mathbf x \in X(d)$, define the coefficient 
\[
    c(\mathbf x) := \frac{(x_1+x_2+\dots +x_{d}-1)!}{x_1! x_2!\dots x_{d}!}.
\]

\begin{corollary}
    \label{cor:2adicdiff}
Suppose $k \in \mathbb N$ and $\mathbf x \in X(2k-1)$.  
Then
\[
\nu_2 \left( c(2\mathbf x) + \frac{(2k-1)}{2}c(\mathbf x)^2 \right ) \geqslant 0.
\]
\end{corollary}
\begin{proof}
    Let $\mathbf x = (x_1,\dots,x_{2k-1})$.
    Since $2k-1$ is odd, there must be some $i$ for which $x_i$ is odd and set $y = x_i$.
    Let $s = x_1 + \dots + x_{2k-1} - y$.
    Note that 
    $$c(2\mathbf x) = \frac{1}{2y}\frac{(2s+2y-1)!}{(2s)!(2y-1)!}\frac{(2s)!(2y)!}{(2x_1)!(2x_2)!\dots (2x_{2k-1})!}$$
    and 
    $$c(\mathbf x) = \frac{1}{y}\frac{(s+y-1)!}{s!(y-1)!}\frac{s!y!}{x_1!x_2!\dots x_{2k-1}!}.$$
    Since the other factors are multinomial coefficients, we see that $c(2\mathbf x) \in \frac{1}{2y}\mathbb Z$ and $c(\mathbf x) \in \frac{1}{y}\mathbb Z$.
    By Lemma~\ref{lem:2adicdiff}, we have
\begin{equation}
    \label{eqn:c1}
    \nu_2 \left( \frac{(2s)!(2y)!}{(2x_1)!(2x_2)!\dots (2x_{2k-1})!} \right ) = \nu_2 \left( \frac{s!y!}{x_1!x_2!\dots x_{2k-1}!} \right ) \geqslant 0.
\end{equation}
    Furthermore, 
    \begin{align*}
    \frac{(s+y-1)!}{s!(y-1)!} \frac{s+y}{y} &= \frac{(s+y)!}{s!y!}  \\
    \frac{(2s+2y-1)!}{(2s)!(2y-1)!} \frac{s+y}{y} &= \frac{(2s+2y)!}{(2s)!(2y)!}.
    \end{align*}
    Therefore, by Lemma~\ref{lem:2adicdiff}, we also have 
 \begin{equation}
     \label{eqn:c2}
         \nu_2 \left( \frac{(s+y-1)!}{s!(y-1)!} \right ) = \nu_2 \left( \frac{(2s+2y-1)!}{(2s)!(2y-1)!} \right ) \geqslant 0.
 \end{equation}
    First suppose that $\nu_2((c(\mathbf x)) \geqslant 1$.
    Using \eqref{eqn:c1} and \eqref{eqn:c2}, it follows that $\nu_2 \left ((c(2\mathbf x) \right ) \geqslant 0$ and hence $\nu_2 \left( c(2\mathbf x) + \frac{(2k-1)}{2}c(\mathbf x)^2 \right ) \geqslant 0$.
    Finally, suppose that $\nu_2((c(\mathbf x)) = 0$.
    Using \eqref{eqn:c1} and \eqref{eqn:c2}, it follows that $\nu_2 \left (c(2\mathbf x) \right ) = \nu_2 \left( \frac{(2k-1)}{2}c(\mathbf x)^2 \right )$ and hence we obtain the statement of the corollary.
\end{proof}

Define $X^\prime(d) := \left \{ \mathbf x \in X(d)  \; | \; x_{d} = 0 \text{ and } x_i = 0 \text{ for all odd $i$} \right \}$.

\begin{theorem}
\label{thm:a_4k-1}
    Let $n \in \mathbb N$, $q > 2$ be a power of $2$, and $H \in \mathcal H_n(q)$.
    Suppose that $\Gamma^{\mathrm{res}}(H)$ is an Euler graph.
    Write $\operatorname{Char}_{H}(x)=\sum_{i=0}^n 
    a_ix^{n-i}$.
    Then, for $k \in \{2,\dots ,\lfloor (n+1)/4 
    \rfloor\}$, we have 
\[
\frac{a_{4k-1}}{(1-\zeta)^{4k-2}} -  (2k-1)\left(\sum_{\mathbf x \in X^\prime(4k-2)} c(\mathbf x) P_{4k-2}(\mathbf x) + \sum_{\mathbf x \in X(2k-1)} \frac{c(\mathbf x)}{1-\zeta} P_{2k-1}(\mathbf x) \right ) \in (1-\zeta) \mathbb Z[\zeta]
\]
where, for $d \in \mathbb N$ and $\mathbf x \in  X(d)$,
$$
P_d(\mathbf x) :=  \prod_{i=1}^{\lfloor d/2 \rfloor} \left( \frac{a_{2i+1}}{(1-\zeta)^{2i}} \right) ^{x_{2i}}  \prod_{i=1}^{\lceil d/2 \rceil} \left( \frac{a_{2i+1}+a_{2i}+a_{2i-1}(a_3 + a_2 + a_1 + n)}{(1-\zeta)^{2i-1}} \right)^{x_{2i-1}}.
$$
\end{theorem}
\begin{proof}
    Let $A = (J-H)/(1-\zeta)$ and write $\operatorname{Char}_A(x)=\sum_{i=0}^n 
    b_ix^{n-i}$.
    By Lemma~\ref{lem:a_kAndb_k},
    \begin{equation}
        \label{b_4k-2}
    a_{4k-1} -  
    (1-\zeta)^{4k-2} b_{4k-2} \in (1-\zeta)^{4k} 
    \mathbb Z[\zeta].
    \end{equation}
    As $4k-2$ is congruent to $2$ modulo $4$,
    by Lemma~\ref{lem:quadtracCong}, we have
\begin{equation}
    \label{eqn:traceRel}
    \frac{\operatorname{tr}(A^{4k-2})}{2} + \frac{\operatorname{tr} (A^{2k-1})^2}{2}-\frac{ \operatorname{tr} (A^{2k-1})}{1-\zeta} \in (1-\zeta) \mathbb Z[\zeta].
\end{equation} 
    Using Newton's identities, for all $i \in \{1,\dots,n\}$, we have 
    \begin{equation}
         \label{eqn:traceForm}
    \operatorname{tr}(A^{i})=  i\sum_{\substack{\mathbf x \in  X(i)}} c(\mathbf x) \prod_{j=1}^{i} (-b_j)^{x_j}.
    \end{equation}

 Define the following subsets of $ X(4k-2)$ as
    \begin{align*}
        X_0 &:= \{ \mathbf x \in  X(4k-2) \; | \; x_i \text{ is even for all $i$}  \}; \\
        X_1 &:= \{ \mathbf x \in  X(4k-2) \; | \; x_i \text{ is odd for some $i$}  \}.
    \end{align*}
    Clearly $ X(4k-2) =  X_0 \cup  X_1$.
    Furthermore, observe that, for each $\mathbf x \in  X_0$, we have $\mathbf x = (2x_1,\dots,2x_{2k-1},0,\dots,0)$, where $(x_1,\dots,x_{2k-1}) \in X(2k-1)$.
     Therefore
     \begin{equation}
         \label{trace(A^4k-2)2}
    \frac{\operatorname{tr}(A^{4k-2})}{2} - (2k-1)\left ( \sum_{\mathbf x \in  X(2k-1)} c(2\mathbf x) \prod_{j=1}^{2k-1} b_j^{2x_j} + \sum_{\substack{\mathbf x \in  X_1}} c(\mathbf x) \prod_{j=1}^{4k-2} (-b_j)^{x_j} \right ) \in (1-\zeta)\mathbb Z[\zeta].
    \end{equation}
  
    By Lemma~\ref{lem:2valbound}, for each $\mathbf x \in  X_1$, we have $\nu_2(c(\mathbf x)) \geqslant 0$.
   Since, for each $\mathbf x \in  X(2k-1)$, there exists an odd $i$ such that $x_i > 0$, using Corollary~\ref{cor:oddbi} and \eqref{eqn:traceForm}, we have
      \begin{align}
         \label{trace(A^2k-1)22}
    \frac{\operatorname{tr}(A^{2k-1})^2}{2} &- (2k-1)^2\sum_{\mathbf x \in  X(2k-1)} \frac{c(\mathbf x)^2}{2}\prod_{j=1}^{2k-1} b_j^{2x_j} \in (1-\zeta)\mathbb Z[\zeta].
    \end{align}
    By Lemma~\ref{lem:2valbound}, for each $\mathbf x \in   X(2k-1)$, we have $\nu_2(c(\mathbf x)) \geqslant 0$.
    By Corollary~\ref{cor:2adicdiff}, for each $\mathbf x \in  X(2k-1)$, the $2$-adic valuation of $(2k-1)c(2\mathbf x) + (2k-1)^2c(\mathbf x)^2/2$ is at least $0$.
    Furthermore, for each $\mathbf x \in  X(2k-1)$, there exists an odd $i$ such that $x_i > 0$.
    Thus, using Corollary~\ref{cor:oddbi} and combining  \eqref{eqn:traceRel}, \eqref{eqn:traceForm}, \eqref{trace(A^4k-2)2}, and \eqref{trace(A^2k-1)22}, we find that
     \begin{align*}
        & -b_{4k-2} +  (2k-1) \left ( \sum_{\mathbf x \in  X^\prime(4k-2)} c(\mathbf x) \prod_{j=1}^{4k-2} (-b_j)^{x_j} + \sum_{\mathbf x \in  X(2k-1)} \frac{c(\mathbf x)}{1-\zeta} \prod_{j=1}^{2k-1} (-b_j)^{x_j} \right )  \in (1-\zeta)\mathbb Z[\zeta].
    \end{align*}
    By Proposition~\ref{pro:idealprelim} (f), we have $2 \in (1-\zeta)\mathbb Z[\zeta]$.
    Furthermore, using Lemma~\ref{lem:a_kAndb_k}, we find that
    \begin{align}
            \label{eqn:b4k-2}
        &  -b_{4k-2} +  (2k-1)\left ( \sum_{\mathbf x \in  X^\prime(4k-2)} c(\mathbf x) P_{4k-2}(\mathbf x) + \sum_{\mathbf x \in  X(2k-1)} \frac{c(\mathbf x)}{1-\zeta} P_{2k-1}(\mathbf x) \right )  \in (1-\zeta)\mathbb Z[\zeta].
    \end{align}
    The theorem then follows from combining \eqref{b_4k-2} and \eqref{eqn:b4k-2}.
\end{proof}

Combine Theorem~\ref{thm:a_4k-1} with Proposition~\ref{pro:realIdeal} to obtain the following corollary.

\begin{corollary}
\label{cor:a_4k-1}
    Let $n \in \mathbb N$, $q > 2$ be a power of $2$, and $H \in \mathcal H_n(q)$.
    Suppose that $\Gamma^{\mathrm{res}}(H)$ is an Euler graph.
    Write $\operatorname{Char}_{H}(x)=\sum_{i=0}^n 
    a_ix^{n-i}$.
    Then, for each $k \in \{2,\dots ,\lfloor (n+1)/4 
    \rfloor\}$, the residue class of $a_{4k-1}$ in $\rho^{\lceil (4k-1)/2 \rceil}\mathbb Z[\zeta+\zeta^{-1}]$ is determined by the values of $a_1, \dots, a_{4k-3}$.
\end{corollary}

Now we are ready to complete the proof of Theorem~\ref{thm:mainBounds}.

\begin{proof}[Proof of Theorem~\ref{thm:mainBounds} (e)]
    Let $H \in \mathcal H_n(2^f)$ where $n$ is odd and $f > 1$ and write $\operatorname{Char}_{H}(x)=\sum_{i=0}^n a_ix^{n-i}$.
    By Remark~\ref{rem:a0a1a2}, $a_0 = 1$ and $a_2$ is determined by the value of $a_1$.
    Since $n$ is odd, $a_1$ must also be an odd integer.
    Using Remark~\ref{rem:rationalElements}, modulo $\rho^e \mathbb Z[\zeta+\zeta^{-1}]$, there are at most $2^{\lceil 2^{2-f}e \rceil - 1}$ residues for $a_1$ modulo $\rho^e \mathbb Z[\zeta+\zeta^{-1}]$.
    By Corollary~\ref{cor:2powerCoeffs} and Remark~\ref{rem:countResidues}, there are at most 
    $2^{e-\lfloor i/2 \rfloor}$ possible residues for $a_i$ modulo $\rho^e \mathbb Z[\zeta+\zeta^{-1}]$, where $i \in \{3,\dots, 2e-1\}$, and just one possible residue for $a_i$ modulo $\rho^e \mathbb Z[\zeta+\zeta^{-1}]$ for $i \geqslant 2e$.
    
    Let $k \in \{2,\dots ,\lfloor (e+1)/4 \rfloor\}$.
    Using Proposition~\ref{pro:uniqueEuler}, without loss of generality, we can assume that the residue graph $\Gamma^{\mathrm{res}}(H)$ is an Euler graph.
    Now apply Corollary~\ref{cor:a_4k-1} to find that, modulo $\rho^{\lceil (4k-1)/2 \rceil}\mathbb Z[\zeta+\zeta^{-1}]$, the residue class of $a_{4k-1}$ is determined by $a_1, \dots ,a_{4k-3}$.
    Set $S= \{3,4,\dots,2e-1\}$ and $T=\{7,11,\dots,4\lfloor (2e+1)/4 \rfloor-1\}$.
    Thus, by inductively fixing the values of $a_3$, $a_4$, and so on, we find that there are at most
    \[
        2^{\lceil 2^{2-f}e \rceil - 1}\prod_{i \in S \backslash T}2^{e-\lfloor i/2 \rfloor}\prod_{i \in T}2^{e-\lceil i/2 \rceil} = 2^{(e-1)^2+\lceil 2^{2-f}e \rceil - 1-\lfloor e/4\rfloor}
    \]
    possible residue classes for $\operatorname{Char}_S(x)$ modulo $(\rho^e \mathbb Z[\zeta+\zeta^{-1}])[x]$.
\end{proof}

\begin{proof}[Proof of Theorem~\ref{thm:mainBounds2} (e)]
    Note that $|\mathcal X_n^\prime(q,e)|$ is equal to the number of congruence classes of $\operatorname{Char}_H(x)$ modulo $\rho^e\mathbb Z[\zeta+\zeta^{-1}][x]$ where $H \in \mathcal H_n(q)$ and the diagonal entries of $H$ are all equal to $1$.
    From this point the proof is similar to corresponding proof of Theorem~\ref{thm:mainBounds}.
\end{proof}

\section{Acknowledgement}

The authors are grateful to the referee for their comments on the manuscript, which have led to various improvements to the exposition.

\bibliographystyle{amsplain}

\end{document}